\documentclass[aos]{imsart}

\RequirePackage{amsthm,amsmath,amsfonts,amssymb}
\RequirePackage{natbib}
\RequirePackage[colorlinks,citecolor=blue,urlcolor=blue]{hyperref}
\RequirePackage{graphicx}
\usepackage{xcolor}
\usepackage{bm}
\usepackage{tikz}
\usetikzlibrary{arrows.meta}
\usepackage{calligra}
\usepackage{mathabx}
\startlocaldefs
\theoremstyle{plain} 
\newtheorem{thm}{Theorem}[section]
\newtheorem{lem}[thm]{Lemma}
\newtheorem{prop}[thm]{Proposition}
\newtheorem{cor}[thm]{Corollary}

\theoremstyle{remark}
\newtheorem{defn}[thm]{Definition}

\newtheorem{ex}[thm]{Example}

\newtheorem{rem}[thm]{Remark}

\DeclareMathOperator{\tr}{tr} 
\DeclareMathOperator{\cov}{\mathbb{V}} 
\DeclareMathOperator{\V}{\mathbb{V}}

\DeclareMathOperator{\dom}{dom} 
\renewcommand{\d}{{\rm d}}
 
\DeclareMathOperator{\conv}{Conv} 

\newcommand{\B}{\mathbb{B}}

\newcommand{\cC}{\mathcal{C}}

\newcommand{\E}{\mathbb{E}}
\newcommand{\cE}{\mathcal{E}}

\newcommand{\cG}{\mathcal{G}}

\renewcommand{\i}{{\rm i}}
\newcommand{\cI}{\mathcal{I}}

\newcommand{\cL}{\mathcal{L}}
\renewcommand{\L}{\mathbb{L}}

\newcommand{\cS}{\mathcal{S}}

\renewcommand{\P}{\mathbb{P}}

\newcommand{\R}{\mathbb{R}}
\renewcommand{\S}{\mathbb{S}}

\newcommand{\cT}{\mathcal{T}}

\newcommand{\cV}{\mathcal{V}}
\newcommand{\X}{\mathbf{X}}

\newcommand{\bs}{\boldsymbol}

\newcommand{\<}{\langle}
\renewcommand{\>}{\rangle}

\newcommand{\g}{\text{\calligra g}\,}
\newcommand{\cconv}{{\rm C}\overline{{\rm onv}}}

\newcommand{\note}[1]{\textbf{\textcolor{red}{#1}}}

\endlocaldefs

\begin{document}

\begin{frontmatter}
\title{Entropic covariance models}
\runtitle{Entropic covariance models}

\begin{aug}
\author{\fnms{Piotr}~\snm{Zwiernik}\ead[label=e2]{piotr.zwiernik@utoronto.ca}\orcid{0000-0003-3431-131X}}
\address{Department of Statistical Sciences,
University of Toronto\printead[presep={,\ }]{e2}}
\end{aug}

\begin{abstract}
In covariance matrix estimation, one of the challenges lies in finding a suitable model and an efficient estimation method. Two commonly used modelling approaches in the literature involve imposing linear restrictions on the covariance matrix or its inverse. Another approach considers linear restrictions on the matrix logarithm of the covariance matrix. In this paper, we present a general framework for linear restrictions on different transformations of the covariance matrix, including the mentioned examples. Our proposed estimation method solves a convex problem and yields an $M$-estimator, allowing for relatively straightforward asymptotic (in general) and finite sample analysis (in the Gaussian case). In particular, we recover standard $\sqrt{n/d}$ rates, where $d$ is the dimension of the underlying model. Our geometric insights allow to extend various recent results in covariance matrix modelling. This includes providing unrestricted parametrizations of the space of correlation matrices, which is alternative to a recent result utilizing the matrix logarithm. \end{abstract}



\end{frontmatter}

\section{Introduction}

Estimating the covariance matrix is a fundamental problem in many fields, including multivariate statistics, spatial statistics, finance, and machine learning. The literature offers a wide range of models that have been considered for this purpose; e.g.  \cite{andersonLinearCovariance, jennrich1986unbalanced, boik2002spectral, pourahmadi2013high}. One popular approach involves exploiting linear restrictions on factors in a decomposition of $\Sigma$ or its transformation \cite{pourahmadiGLM}. For instance, in linear structural equation models, specific entries of the matrix $L$ in the LDL decomposition $\Sigma^{-1} = LDL^\top$ are set to zero.

Since modelling via the LDL decomposition heavily relies on the variable ordering in the system, as an alternative, linear restrictions can be directly imposed on the covariance matrix $\Sigma$ or some transformation of it; e.g.  \cite{andersonLinearCovariance, dempster1972covariance, sturmfels2010multivariate}. This approach has gained attention due to the prevalence of such structures in multiple applications. Examples include Toeplitz matrices or block-Toeplitz matrices in time series and spatial statistics \cite{anderson1978maximum}, linear structures encoded by trees in Brownian motion tree models \cite{zwiernik2017maximum}, and other types of symmetries \cite{szatrowski2004patterned}. Gaussian graphical models, which enforce sparsity on the inverse of $\Sigma$, and their colored versions have also been widely used in multivariate statistics and machine learning \cite{dempster1972covariance,lau96,hojsgaard2008graphical}. These models are popular due to their interpretability, as the entries of $\Sigma$ and its inverse correspond after normalization to correlations or partial correlations.

Another type of restriction considered in the literature involves linear constraints on the matrix logarithm of the covariance matrix \cite{leonard1992bayesian, chiu1996matrix, battey17}. While the interpretation of such constraints is generally less clear, these models have gained popularity because modelling the matrix logarithm $\log(\Sigma)$ does not require handling the positivity constraints. When $\Sigma$ is diagonal, these models can be viewed as extensions of classical log-linear models for heterogeneous variances. The matrix logarithm of the covariance matrix has found applications in stochastic volatility models, medical imaging, spatial statistics, and quantum geometry  \cite{kawakatsu2006matrix, ishihara2016matrix, asai2015long, bauer2011forecasting, zhu2009intrinsic, lesage2007matrix, pavlov2022gibbs}.

Notably, the fact that the matrix logarithm of a covariance matrix is an unrestricted symmetric matrix has important theoretical implications. For instance, in recent work \cite{archakov2021new} have shown that the logarithm of the \emph{correlation} matrix has properties similar to Fisher's Z-transformation. Having an unrestricted parametrization of correlation matrices is considered a major breakthrough in temporal modelling of correlation matrices, which is critical in the GARCH approach. Leveraging similar ideas, we provide technology to get an unrestricted parameterization of various other covariance models, which may be of independent interest. In this context, we propose to study the mapping $\Sigma\mapsto \Sigma-\Sigma^{-1}$ as a more tractable alternative to the matrix logarithm.

\textbf{Main goals of this paper: }
The research on the matrix logarithm has motivated the exploration of linear restrictions on more general functions of the covariance matrix. In this statistical context, functions such as the matrix logarithm are treated as link functions, analogous to generalized linear models (GLMs) \cite{pourahmadiGLM,zou2017covariance,lin2023additive}. Our paper contributes to the development of a GLM methodology and a data-based framework for modelling covariance matrices, building on the work initiated by Anderson, Pourahmadi, and others. In our general approach we follow  steps of  \cite{barndorff:78,banerjee2005clustering} where convex analysis was used to study theoretical properties of exponential families. This is similar to the construction of generalized exponential families and matrix nearness problems \cite{dawidGEF,dhillon2008matrix}, albeit applied to covariance matrices in a broader sense than previously explored.

The main goals of this paper are twofold:
\begin{enumerate}
\item[(i)] To propose a general framework for modelling covariance matrices that allows for efficient estimation procedures based on convex optimization.
\item[(ii)] To enhance the understanding of the geometry of covariance matrices and show how convexity simplifies the statistical analysis.
\end{enumerate}
We now briefly describe how these goals are approached.

\textbf{Entropic covariance models (informal):}
The key idea is to utilize the gradient mapping $\nabla F(\Sigma)$ of a general strictly convex and differentiable function defined on the positive definite cone. This mapping induces a one-to-one transformation of covariance matrices, and we impose affine (or general convex) restrictions on the result of this transformation. Affine constraints can arise from regression of the covariance matrix (or its transformation) on auxiliary information \cite{zou2017covariance,lin2023additive}, specific symmetry patterns \cite{szatrowski2004patterned}, or sparsity \cite{dempster1972covariance,hastie2015statistical}.

For a concrete example, consider the set of $3\times 3$ covariance matrices $\Sigma$ such that $L=\log(\Sigma)$ satisfies $L_{13}=0$. This example appears later in Section~\ref{sec:sparse}.  Models with zero restrictions on $\log\Sigma$ have been recently studied in \cite{rybak2021sparsity,pavlov2023logarithmically}. Here the main problem is that  simple constraints in $L$ translate to complicated constraint in $\Sigma$, which may potentially complicate the estimation process. This is one of the problems we  address in this article.

\textbf{Estimation under linear constraints:}
Estimating models under linear restrictions on the inverse covariance matrix is relatively straightforward. Let $\X\in \R^{n\times m}$ be the data matrix with independent rows coming from a mean zero distribution with covariance $\Sigma_0$. A natural approach is to optimize the Gaussian log-likelihood, which, up to additive constants, is given by
\begin{equation}\label{eq:GLL}
\ell(K) = \tfrac{n}{2}\log\det(K) - \tfrac{n}{2}\tr(S_n K), \qquad S_n = \frac{1}{n} \X^\top \X,
\end{equation}
where $K=\Sigma^{-1}$ satisfies the given constraints \cite{andersonLinearCovariance, sturmfels2010multivariate, barratt2022covariance}. This optimization problem is convex since $\ell(K)$ is a strictly concave function of $K$, and the constraints are linear in $K$. However, in the case of linear constraints on $\Sigma$, a canonical estimation approach is less obvious, and the Gaussian log-likelihood becomes multimodal \cite{zwiernik2017maximum}. This has motivated research on alternative estimation methods \cite{anderson73, christensen:89, sturmfels2019estimating,amendola2021likelihood}. One natural approach is to replace the maximum likelihood estimator (MLE) with the least squares estimator, which minimizes the Frobenius distance $\|\Sigma-S_n\|^2_F$ over all $\Sigma$ in the given linear subspace.

In this paper, by utilizing the Bregman divergence \cite{bregman1967relaxation}, we generalize both the least squares approach for estimating under linear restrictions on $\Sigma$ and the problem of minimizing $\ell(K)$ for restrictions that are linear in $K$. Bregman matrix divergences have been used for matrix estimation and matrix approximation problems; e.g. \cite{dhillon2008matrix,ravikumar2010high,cai2012optimal,llorens2022projected}. However, in these papers, Bregman divergence was used to analyze existing covariance models. Here, it is studied in the context of new models and to provide more insight in covariance matrix geometry.

 For every entropic model the resulting loss function is strictly convex, and its Hessian does not depend on the data. This is similar to the negative log-likelihood function in exponential families, making the theoretical analysis of this estimator rather straightforward following the elegant work of \cite{niemiro1992asymptotics}. In particular, we show that our proposed estimator, which we call the Bregman estimator, is consistent and asymptotically normal when the sample size goes to infinity. We complementary these results with finite sample analysis in the case when the data come from the Gaussian distribution, which shows that the Bregman estimator remains a good estimator in high-dimensions as long as  $\sqrt{{d}/{n}}$ remains small, where $d$ is the dimension of the underlying model. This recovers standard parametric rates. Further, in Section~\ref{sec:numeric} we provide various approaches to perform numerical optimization for Bregman estimators.

\textbf{Geometry of entropic models:} One of the contributions of this paper is providing new insights into the geometry of various covariance models that explain and sometimes greatly generalize existing results. One example is a far reaching generalization of Theorem~1 in \cite{archakov2021new}, which provides an unrestricted parametrization for correlation matrices; see Section~\ref{sec:archakov1}. Another class of insights is related with the Jordan algebras of symmetric matrices, which we discuss in Section~\ref{sec:jordan}. These results explain why for certain entropic models the maximum likelihood estimator is available in closed form; e.g. \cite{lesage2007matrix}. We expect these results will greatly improve our understanding of the linear models on the matrix logarithm of $\Sigma$ and other entropic models.

The article is organized as follows. Section~\ref{sec:defs} provides a brief overview of the necessary background in convex analysis and introduces the main definitions and running examples; see also Appendix~\ref{sec:spec}. Our proposed estimation method is presented in Section~\ref{sec:breg}, where we also demonstrate the convexity of the underlying optimization problem. Section~\ref{sec:stat} presents basic statistical analyses of the resulting estimator. Section~\ref{sec:mixed} investigates the geometry of entropic models in connection with convex analysis and mixed parametrizations. Section~\ref{sec:numeric} proposes simple numerical algorithms for computing our estimator. Section~\ref{sec:sparse} focuses on sparsity patterns.  Section~\ref{sec:jordan} provides insights into a particular type of linear constraints that define Jordan algebras. We conclude the paper in Section~\ref{sec:other} with a short discussion of future research directions.


\section{Preliminaries and definitions}\label{sec:defs}

Let $\mathbb{S}^m$ denote the real vector space of symmetric $m\times m$ matrices, and let $\mathbb{S}^m_+$ and $\overline{\mathbb{S}}^m_+$ represent the subsets of matrices that are positive definite and positive semidefinite, respectively. We equip $\mathbb{S}^m$ with the standard inner product $\langle A,B\rangle = \tr(AB)$ and the induced Frobenius norm $\|A\|_F = \sqrt{\langle A,A\rangle}$.

\subsection{Convex functions on $\S^m$}
By $\conv(\S^m)$ denote the set of convex functions $F:\S^m\to \R\cup \{+\infty\}$ that are not identically equal to $+\infty$ (these are called sometimes proper convex functions). The domain of $F\in \conv(\S^m)$ is the nonempty set $\dom( F)=\{A\in \S^m: F(A)<+\infty\}$. By $\cconv(\S^m)$ denote the class of all functions in $\conv(\S^m)$ that are lower semicontinuous on $\S^m$ --- these are also known as closed convex functions. Recall that $F$ is lower semicontinuous if the lower level-set $\{A: F(A)\leq t\}$ is closed for all $t\in \R$. Because convex functions are always continuous in the interior of their domain, this involves conditions on how $F$ behaves on the boundary of $\dom F$; see Appendix~\ref{sec:basicconv} for more details and references.

Reserve a special notation $\cE^m$ for functions $F:\S^m\to \R\cup\{+\infty\}$ such that:
\begin{enumerate}
	\item [(i)] $F\in \cconv(\S^m)$,
	\item [(ii)] $\S^m_+\subseteq \dom(F)\subseteq \overline\S^m_+$ or equivalently ${\rm int}(\dom(F))=\S^m_+$,
	\item [(iii)] $F$ is strictly convex and continuously differentiable on $\S^m_+$.
\end{enumerate}

\begin{rem}\label{rem:addi}
If $F$ satisfies only (i) and (iii), we replace $F(A)$ with $$
F(A)+\i_{\overline{\S}^m_+}(A), \qquad\mbox{where}\qquad \i_{\overline{\S}^m_+}(A)\;=\;\begin{cases}
	0 & \mbox{if }A\in \overline{\S}^m_+,\\
	+\infty & \mbox{otherwise}.
\end{cases}
$$
Because $\overline{\S}^m_+$ is closed and convex,  $\i_{\overline{\S}^m_+}\in \cconv(\S^m)$, and so also $F(A)+\i_{\overline{\S}^m_+}(A)\in \cconv(\S^m)$. Moreover, the interior of the domain of $F(A)+\i_{\overline{\S}^m_+}(A)$ is $\S^m_+$. For brevity we often include the indicator function only implicitly.
\end{rem}

In our running examples below, we define the function $F(\Sigma)$ for $\Sigma\in \S^m_+$. The function is then extended by lower semicontinuity to the boundary of $\overline \S^m_+$, and it is equal to $+\infty$ for all other points. The following are our running examples:
$$
\begin{array}{lp{2cm}l}
	\textbf{ (A)}\quad F_A(\Sigma)=
		-\log\det\Sigma & &\textbf{ (B)}\quad F_B(\Sigma)=\frac12 \tr(\Sigma^2)\\[.3cm]
		\textbf{ (C)}\quad F_C(\Sigma)=-\tr(\Sigma-\Sigma\log\Sigma) & &\textbf{ (D)}\quad F_D(\Sigma)=\tr(\Sigma^{-1})
\end{array}
$$
and the details on the lower semicontinuous extension are provided in Appendix~\ref{sec:runex}. The function $-F_A$ is called the Gaussian entropy. The function $2F_B$ is the squared Frobenius norm of $\Sigma$. The function $-F_C$ is the von Neumann entropy. Both (B) and (D) can be easily generalized:
$$
\begin{array}{lp{1cm}l}
	\textbf{ (B')}\quad F_{B,p}(\Sigma)=\frac1p \tr(\Sigma^p)\;\;\; \mbox{for }p\geq 1 & &\textbf{ (D')}\quad F_{D,p}(\Sigma)=\frac1p\tr(\Sigma^{-p})\;\;\; \mbox{for }p\geq 0.\end{array}
$$
Note that $p F_{B,p}$ is the p-Schatten norm of $\Sigma$ raised to power $p$, $F_{B,p}(\Sigma)=\tfrac1p\|\Sigma\|^p_p$, where denoting by $\lambda_1,\ldots,\lambda_m$ the eigenvalues of $\Sigma\in \S^m_+$ we have
\begin{equation}\label{eq:schatten}
	\|\Sigma\|_p=\sqrt[p]{\lambda_1^p+\ldots+\lambda_m^p}.
\end{equation}

Although we keep our theory as general as possible, we note that all our running examples are spectral functions of the form $F(\Sigma)=\tr(\phi(\Sigma))$ where $\phi:\R\to\R$ and the notation $\phi(\Sigma)$ means the corresponding matrix function, that is, if $\Sigma=U\Lambda U^\top$ is the spectral decomposition of $\Sigma$ then $\phi(\Sigma)=U \phi(\Lambda) U^\top$, where $\phi$ in $\phi(\Lambda)$ is applied to each diagonal entry of $\Lambda$; see, e.g., \cite{higham}. We call such functions spectral sums and, in Appendix~\ref{sec:spec}, we present useful results which make working with this function class particularly easy.

\subsection{Entropic covariance models}
 We are now ready to describe our set-up. The main object in our analysis is the gradient $\nabla F$ of $F$. We start with the following well-known fact.
 \begin{lem}\label{lem:11}
 	If $F\in \cE^m$, then $\nabla F:\S^m_+\to \S^m$ defines a one-to-one function on $\S^m_+$.
 \end{lem}
 \begin{proof}
 	Fix $A\in \S^m$ and note that the function
 $\<A,\Sigma\>-F(\Sigma)$ is strictly concave and continuously differentiable on $\S^m_+$. Thus, if its maximum in $\S^m_+$ exists, it must be a stationary point and so it must satisfy $\nabla F(\Sigma)=A$. By strict convexity it must be the unique such point. This shows that for each $A\in \S^m$ there exists at most one point $\Sigma\in \S^m_+$ such that $\nabla F(\Sigma)=A$. 
 \end{proof} 
 
Our modelling technique is to apply the transformation $\nabla F(\Sigma)$ and impose convex restrictions on it. It is useful to denote
\begin{equation}\label{eq:Lplus}
\mathbb L^m_+\;:=\;\nabla F(\S^m_+).
\end{equation}
Lemma~\ref{lem:11} motivates the following definition.
 \begin{defn}[Linear Entropic Covariance Model]\label{def:entropicmodel}
 Fix an affine subspace $\cL\subseteq \S^m$. The corresponding {linear entropic model} is 
 $$
 M_F(\cL)\;:=\;\{\Sigma\in \S^m_+: \;\nabla F(\Sigma)\in \cL\}.
 $$
 More generally, we define $M_F(\cC)$ for any general closed convex constraints $\cC\subseteq \S^m$.
 \end{defn}
 
 \begin{rem}
 	As we argue in this paper, this definition unifies many well known examples and leads to some new ones. Not all proposed models have a clear statistical interpretation but they may still be useful. For example, the matrix logarithm allows us to regress the covariance matrix on external data \cite{chiu1996matrix}. So here $\cL$ will be not fixed by the modeller but will be generated by the external data. 
 \end{rem}

 In the definition we implicitly assumed that $M_F(\cL)$ is non-empty. This is a recurrent assumption of this paper.
 \begin{enumerate}
	\item [{\sc Assumption 1:}] The mapping $F\in \cE^m$ and the subspace $\cL\subseteq \S^m$ satisfy $\cL\cap \L^m_+\neq \emptyset$.
\end{enumerate}
 
 Some interesting examples of the function $F$ are given by popular matrix entropy functions like the negative Gaussian entropy (A) and the negative von Neumann entropy (C) (hence the name). More examples will be discussed later in the paper. For now, $F$ is relevant for us only through the mapping $\nabla F(\Sigma)$, which defines a suitable reparametrization of the covariance matrix. In the following examples, we refer to Proposition~\ref{prop:Fder} for a simple technique to compute $\nabla F(\Sigma)=\nabla \tr(\phi(\Sigma))$ by computing the derivative of $\phi$: 
 $$
 \mbox{if }F(\Sigma)=\tr(\phi(\Sigma))\qquad \mbox{then}\quad\nabla F(\Sigma)=\phi'(\Sigma).
 $$

  \begin{ex}[A]\label{ex:GEntr}
Since $-\log\det(\Sigma)=-\tr(\log(\Sigma))$, here $\phi_A(x)=-\log(x)$ if $x>0$ and so $\phi_A'(x)=-\tfrac1x$. By Proposition~\ref{prop:Fder}, $\nabla F_A(\Sigma)=\phi_A'(\Sigma)=-\Sigma^{-1}$. The model $M_{F_A}(\cL)$ is described by all $\Sigma\in \S^m_+$ such that $-\Sigma^{-1}\in \mathcal L$. Here $\L^m_+=-\S^m_+$. Models of this form are classically known in statistics as linear concentration models \cite{dempster1972covariance,anderson73}; see also the introduction for a more comprehensive literature overview. 
\end{ex}

Although linear restrictions on the inverse covariance matrix have many applications, there are important areas (e.g. signal processing) where it is more natural to impose linear symmetry restrictions directly on the covariance matrix. This leads to our second example.
\begin{ex}[B]\label{ex:Frob}
Here $\phi_B(x)=\tfrac12 x^2$ for $x>0$ and so $\phi'_B(x)=x$. This gives $\nabla F_B(\Sigma)=\phi_B'(\Sigma)=\Sigma$. The corresponding entropic model is given by all $\Sigma\in \S^m_+$ such that $\Sigma\in \cL$. Here $\L^m_+=\S^m_+$. This imposes linear restrictions on the covariance matrix as discussed in the introduction. This example can be generalized to the p-th Schatten norm of $\Sigma\in \S^m_+$, $\nabla F_{B,p}(\Sigma)=\Sigma^{p-1}$, which allows us to model linear restrictions on an arbitrary positive power of $\Sigma$. 
\end{ex}

This setting is however completely general and, as we argue below, it is a natural framework to discuss the generalized models for covariance matrices \cite{pourahmadi2000maximum,zou2017covariance,lin2023additive}. Example (C) again links to a model well studied in the literature.
\begin{ex}[C]\label{ex:vN}
We have $\phi_C(x)=-x(1-\log(x))$ for $x>0$ and so $\phi_C'(x)=\log(x)$. Thus, $\nabla F_C(\Sigma)=\phi'_C(\Sigma)=\log(\Sigma)$ and the model is given by all $\Sigma$ such that $\log(\Sigma)\in \cL$. This imposes linear restrictions on the matrix logarithm of the covariance matrix. Although such linear restrictions are hard to interpret statistically, one of the reasons, why this model is useful is because every matrix $L\in \S^m$ is a matrix logarithm of some $\Sigma\in \S^m_+$. In other words, $\L^m_+=\S^m$ and so $\log(\Sigma)$ can be suitably regressed on external information \cite{chiu1996matrix}. In the introduction we provide an extensive literature overview for this model. In Section~\ref{sec:mixed} and in Section~\ref{sec:other} we discuss  alternatives to this transformation with much better algebraic properties.
\end{ex}

 We discuss one more example whose relevance will be explained later. 
\begin{ex}[D]\label{ex:D1} In this example, we have $\phi_D(x)=\tfrac1x$ for $x>0$ and so $\phi_D'(x)=-\tfrac1x$. This gives $\nabla F_D(\Sigma)=\phi_D'(\Sigma)=- \Sigma^{-2}$ and the model is given by all $\Sigma\in \S^m_+$ such that $\Sigma^{-2}\in \cL$. This example has a straightforward generalization: $\nabla F_{D,p}(\Sigma)=-\Sigma^{-p-1}$ for any $p\geq 0$, which allows to impose linear restrictions on powers of $\Sigma^{-1}$. In Section~\ref{sec:other} we briefly motivate such models with zero restrictions. In Example~\ref{ex:sar} we motivate zero restrictions in $\Sigma^{-1/2}$, which, as we argue later, is a very closely related model.
\end{ex}

All our examples are summarized in Table~\ref{tab:1}. The gradient is provided in the second column and the other columns will be discussed in detail later. 
 
 \begin{table}
 	$$
 	\begin{array}{c|c|c|c}
 		F(\Sigma)&\nabla F(\Sigma) & F^*(L) & D_F(S,L)\\
 		\hline
 		-\log\det(\Sigma) & -\Sigma^{-1}& -m-\log\det(-L)& -\log\det((-L)S)+\tr((-L)S-I_m)\\[.2cm]
 		\tfrac12\|\Sigma\|_F^2 &  \Sigma & \tfrac12\tr(L^2)& \tfrac12\|L-S\|^2_F \\[.2cm]
 		\tr(\Sigma^{-1}) & -\Sigma^{-2} & -2\tr(\sqrt{-L})& \tr(S^{-1})-2\tr(\sqrt{-L})+\tr((-L)S)\\[.2cm]
 		-\tr(\Sigma-\Sigma\log(\Sigma)) & \log(\Sigma)& \tr(e^L)& -\tr(S-S\log(S))+\tr(e^L)-\tr(LS) \\[.2cm]
 		\tfrac1p\|\Sigma\|^p_p, p\geq 1 &  \Sigma^{p-1} & \tfrac1q\|L\|_q^q, \;q=\tfrac{p}{p-1}& \tfrac1p\|S\|^p_p+\tfrac{1}{q}\|L\|^q_q-\tr(LS) \\[.2cm]
 		\tfrac1p\|\Sigma^{-1}\|^p_p, p> 0 &  -\Sigma^{-p-1} & -\tfrac1q\|-L\|^q_q,\; q=\tfrac{p}{p+1}& \tfrac1p\|S^{-1}\|^p_p-\tfrac{1}{q}\|-L\|^{q}_q+\tr((-L)S) \\
 		\hline
 	\end{array}
 	$$
 	\caption{Our running examples with the corresponding gradients, conjugate functions, and the Bregman divergence.}\label{tab:1}
 \end{table}

\subsection{Dual construction of $M_F(\cL)$}\label{sec:dual} We note the following dual construction. Suppose $\pi:\mathbb S^m\to \R^d$, for some $d\geq 1$ is an affine function. In the spirit of (generalized) exponential families \cite{barndorff:78,dawidGEF}, we refer to $\pi$ as a sufficient statistics. Given $b\in \R^d$, consider the optimization problem
\begin{equation}\label{ex:OPT}
{\rm minimize}\;\; F(\Sigma)-\<A_0,\Sigma\>\qquad \mbox{subject to }\pi(\Sigma)=b,
\end{equation}
where $A_0\in \S^m$ is a fixed matrix.  If $F\in \cE^m$ then this problem has at most one optimal solution in $\S^m_+$. Since the set $\S^m_+\cap \{\Sigma: \pi(\Sigma)=b\}$ is relatively open (in the affine subspace $\{\Sigma: \pi(\Sigma)=b\}$), this optimal point $\widehat \Sigma$, if it exists, must satisfy the regular first order conditions: we must have  $\pi(\widehat \Sigma)=b$ and, for every $U\in \S^m$ such that $\pi(U)=0$, it must hold that $\<\nabla F(\widehat \Sigma)-A_0,U\>=0$, that is, the directional derivatives in all permitted directions must be zero.  In other words, 
\begin{equation}\label{eq:opt}
\nabla F(\Sigma)-A_0\;\in\;{\rm ker}(\pi)^\perp\;=:\;\cL_0. 
\end{equation}
Note that this equation and the affine space $\cL:=A_0+\cL_0$ do not depend on the vector $b$ and so the condition $\nabla F(\Sigma)\in \cL$ describes all such potential optimizers. 
In fact, a point $\widehat \Sigma\in\S^m_+$ solves \eqref{ex:OPT} for some $b\in \R^d$ if and only if $\nabla F(\widehat \Sigma)\in \cL$. Indeed, if $\nabla F(\widehat \Sigma)\in \cL$ then take $b:=\pi(\widehat \Sigma)$ so that  $\widehat \Sigma$ is an optimum of \eqref{ex:OPT} for this $b$.

If $\cL$ is given by zero restrictions on some coordinates, we get an explicit link to positive definite completion problems. 

\begin{ex}[Positive definite completion] Fix a graph $G$ on $m$ nodes and edge set $E$. We allow $G$ to have self-loops that is edges from $i$ to $i$. Let $\pi:\S^m\to \R^{|E|}$ be given by
$\pi(\Sigma)\;=\;((\Sigma_{ij})_{ij\in E})$. In this case $\ker(\pi)$ is the set of symmetric matrices with zeros on the entries corresponding to the edges of $G$. Thus, $\cL_0$ is the set of symmetric matrices with zero entries for all non-edges of $G$: 
$\cL_0\;=\;\{L\in \S^m:\;\;L_{ij}=0\mbox{ if }ij\notin E\}$.  
Given $S\in \S^m_+$, the solution to \eqref{eq:opt} with $b:=\pi(S)$ and $A_0=0$ is the unique matrix $\widehat \Sigma$ that agrees with $S$ on all the entries $ij\in E$ and such that $\widehat L=\nabla F(\widehat \Sigma)$ is zero on the complementary entries.
\end{ex}


\section{The Bregman estimator}\label{sec:breg}

Consider data $X_1,\ldots,X_n$ from a centered distribution with a covariance matrix $\Sigma_0$ in an entropic model $M_F(\cL)$. Throughout the paper we assume that the true covariance matrix $\Sigma_0\in \S^m_+$ lies in the given model.
\begin{enumerate}
	\item [{\sc Assumption 2:}] $\nabla F(\Sigma_0)\in \cL\cap \L^m_+$.  \end{enumerate}
We can estimate the covariance matrix using the Gaussian log-likelihood \eqref{eq:GLL}. In the case of non-Gaussian data, this log-likelihood is considered as one of the suitable loss functions. This gives an asymptotically statistically optimal procedure as long as all fourth order cumulants of the underlying distribution vanish \cite{browne1974generalized}. For the linear concentration model in Example~\ref{ex:GEntr}  this approach is canonical not only because it leads to an efficient estimator but also because it requires solving a convex optimization problem. Indeed, the Gaussian log-likelihood in \eqref{eq:GLL} is a strictly concave function in $K=\Sigma^{-1}$.

The problem for the general entropic model $M_F(\cL)$ is that optimizing the Gaussian log-likelihood may be quite complicated since the linear constraints in $L=\nabla F(\Sigma)$ translate into non-linear constraints in $K$. This observation has motivated a lot of research on the estimation of the linear covariance models. One valid solution is to use the dual MLE, which provides an efficient alternative to the MLE \cite{christensen:89,kauermann1996dualization,lauritzen2022locally,amendola2021likelihood}. The least squares estimator or generalized least squares estimator has also been used \cite{browne1974generalized}.

\subsection{Definition of the Bregman estimator}
In this section, we propose an estimation procedure for linear entropic models, which offers a natural alternative to the MLE. It generalizes the use of Gaussian likelihood for linear concentration models and the least squares estimation for linear covariance models. An important ingredient of our statistical analysis is the Bregman divergence:
\begin{equation}
D_F(S,\Sigma)\;=\;F(S)-F(\Sigma)-\<\nabla F(\Sigma),S-\Sigma\>,\qquad \Sigma\in \S^m_+, S\in \S^m.
\end{equation}
Note that one of the characterizations of strict convexity for differentiable functions over $\S^m_+$ assures that $D_F(S,\Sigma)\geq 0$ for all $\Sigma\in \S^m_+, S\in \S^m$ with equality if and only if $S=\Sigma$.

Let $S_n$ be the sample covariance matrix defined in \eqref{eq:GLL}. Our proposed estimator is obtained by minimizing the Bregman divergence $D_F(S_n,\Sigma)$ over the entropic model $M_F(\cL)$.  
\begin{defn}
	The Bregman estimator $\widehat \Sigma$ (if it exists) is the global minimizer of $D_F(S_n,\Sigma)$ subject to $\nabla F(\Sigma)\in \cL\cap \L^m_+$. 
\end{defn}

There are two crucial aspects regarding the underlying optimization problem that we will formally state later in this section. Firstly, in Section~\ref{sec:conv}, we demonstrate that $D_F(S_n,\Sigma)$ is a strictly convex function of $L=\nabla F(\Sigma)\in \L^m_+$. Secondly, in Theorem~\ref{th:optexist}, we establish that if the gradient of $F(\Sigma)$ diverges when $\Sigma$ approaches the boundary of $\S^m_+$, then the optimum always exists whenever $S_n\in \S^m_+$. In such cases, there is no explicit need to impose the restriction $\nabla F(\Sigma)\in \L^m_+$ as first-order optimization methods will naturally remain within $\S^m_+$. Before formally proving these assertions, we will examine some examples.

\begin{ex}[A]
In the Gaussian entropy example we have
$$
D_{F_A}(S_n,\Sigma)\;=\;-\log\det(S_n\Sigma^{-1})+\tr(S_n\Sigma^{-1}-I_m),
$$
which is just the standard Kullback-Leibler divergence between two mean zero Gaussian distributions, with covariances $S_n$ and $\Sigma$ respectively. \end{ex}

\begin{ex}[B]
	In the Frobenius norm example we have
	$$
	D_{F_B}(S_n,\Sigma)\;=\;F_B(S_n)-F_B(\Sigma)-\<\Sigma,S_n-\Sigma\>\;=\;\tfrac12 \|\Sigma-S_n\|^2_F.
	$$
	Thus, minimizing $D_{F_B}(S_n,\Sigma)$ over $\Sigma\in \cL$ simply gives the orthogonal projection of $S_n$ on $\cL$ if this projection is positive definite. \end{ex}

The next example proposes a new way of estimating parameters in models that are linear in $\log(\Sigma)$. This provides an alternative to the maximum likelihood estimation considered in \cite{chiu1996matrix}.

\begin{ex}[C]
In the von Neumann case we have
\begin{eqnarray*}
D_{F_C}(S_n,\Sigma)&=&-\tr(S_n-S_n\log(S_n))+\tr(\Sigma-\Sigma\log(\Sigma))-\<\log(\Sigma),S_n-\Sigma\>\\
&=&	-\tr(S_n-S_n\log(S_n))+\tr(\Sigma-S_n\log \Sigma).
\end{eqnarray*}
 \end{ex}

The next example provides some curious connections with the empirical score matching loss. 

\begin{ex}[D]\label{ex:D2}
In our last example given by $F_D(\Sigma)=\tr(\Sigma^{-1})$, we have
	\begin{eqnarray*}
	D_{F_D}(S_n,\Sigma)&=&\tr(S_n^{-1})-\tr(\Sigma^{-1})+\<\Sigma^{-2},S_n-\Sigma\>	\\
	&=& \tr(S_n^{-1}) -2\tr(\Sigma^{-1}-\Sigma^{-1}S_n\Sigma^{-1}).
	\end{eqnarray*}
	This function is convex in $\Sigma^{-1}$ and it corresponds to the empirical score matching loss of \cite{lin2016estimation}. In the next section we also argue that this function is convex in $\Sigma^{-2}$.
\end{ex}

\subsection{Convexity of the underlying optimization problem}\label{sec:conv}

We next analyze $D_F(S_n,\Sigma)$ as a function of $L=\nabla F(\Sigma)$. The conjugate of $F(\Sigma)$ is
\begin{equation}\label{eq:CD}
    F^*(L)\;:=\;\sup_{\Sigma\in \mathbb S^m} \{\<\Sigma,L\>-F(\Sigma)\}.	
\end{equation}
The third column of Table~\ref{tab:1} contains the convex conjugates for our leading examples. Note that we use notation $\|L\|_q$ introduced in \eqref{eq:schatten} also if $q<1$, in which case this is formally not a norm. Explicit calculations for spectral functions could be done by utilizing Theorem~2.3 in \cite{lewis1996convex}. For the special case of spectral sums we use Lemma~\ref{lem:spectrlcconv}: if $F(\Sigma)=\tr(\phi(\Sigma))$ then $F^*(L)=\tr(\phi^*(L))$.
    
     In the next proposition we collect several known results.
\begin{prop}\label{prop:Fstar}
	If $F\in \cE^m$ then (a) $F^*\in \cconv(\S^m)$, (b) $F^*$ is continuously differentiable on ${\rm int}\dom(F^*)$, (c) $\L^m_+\subseteq \dom(F^*)$, (d) $\nabla F$ and $\nabla F^*$ are inverses of each other on $\S^m_+$ and $\L^m_+$. 
\end{prop}    
    \begin{proof}
    Statements (a), (b) follow from Theorem~E.1.1.2,  Theorem~E.4.1.1 in \cite{convex}. To prove (c), note that if $S\in \S^m_+$ and $L=\nabla F(S)$ then $\<\Sigma,L\>-F(\Sigma)$ is strictly concave and well-defined over $\S^m_+$. Since $S$ is a stationary point it must be the optimum and so $F^*(L)<\infty$. Now (d) follows from Corollary~E.1.4.4 in \cite{convex}.
    \end{proof}


Applying Proposition~3.2 in \cite{bauschke1997legendre}, we get that for all $S\in \S^m$ and $\Sigma\in \S^m_+$
$$
D_F(S,\Sigma)\;=\;F(S)+F^*(\nabla F(\Sigma))-\<\nabla F(\Sigma),S\>.
$$
Let $L=\nabla F(\Sigma)$, or equivalently by Proposition~\ref{prop:Fstar}(d), $\Sigma=\nabla F^*(L)$. Slightly abusing notation, from now on, we will write $D_F(S,L)$ to refer to $D_F(S,\nabla F^*(L))$.	This notation is also used in the last column of Table~\ref{tab:1}, where the corresponding Bregman divergences computed above are expressed in terms of $L$. We thus have,
\begin{equation}\label{eq:DFinL}
D_F(S,L)\;=\;F(S)+F^*(L)-\<L,S\>.	
\end{equation}
In particular, $D_F(S,L)$ is strictly convex and differentiable both in $S\in \S^m_+$ and $L\in \L^m_+$.


If the sample covariance matrix $S_n$ is not positive definite, the divergence $D_F(S_n,L)$ may not be finite. In analogy to the log-likelihood function, we  solve instead 
\begin{equation}\label{eq:mainmax}
{\rm maximize}\quad \g_n(L)\;:=\;-F^*(L)+\<L,S_n\>\quad\mbox{subject to }\quad L\in \cL\cap \L^m_+.    
\end{equation}
\noindent We easily see that the gradient of $\g_n$ is well defined on $\L^m_+$ and
\begin{equation}\label{eq:nablag}
\nabla \g_n(L)=-\nabla F^*(L)+S_n.
\end{equation}
The KKT conditions are easy to obtain.
\begin{thm}\label{th:KKTg}Suppose that  $\L^m_+$ is open. The optimum in \eqref{eq:mainmax}, if it exists,  is uniquely given by the pair $(\widehat \Sigma,\widehat L)\in \S^m_+\times \L^m_+$ with $\widehat L=\nabla F(\widehat \Sigma)$ satisfying 
\begin{equation}\label{eq:KKTL}
\widehat L\in \cL\quad\mbox{and}\quad \widehat \Sigma-S_n\in \cL^\perp.	
\end{equation}
\end{thm}

\begin{proof}
    First note that  $\nabla F^*(\widehat L)=\widehat \Sigma$ by Proposition~\ref{prop:Fstar}(d). Since $\L^m_+$ is open,  $\cL\cap \L^m_+$ is relatively open, and so $\widehat L\in \cL\cap \L^m_+$ is optimal if and only if the gradient $\nabla\g_n(\widehat L)=-\widehat\Sigma+S_n$ is orthogonal to $\mathcal L$. 
\end{proof}

In certain situations, it may be easier to solve the dual problem. 

\begin{prop}\label{cor:dual}
Let $A_0\in \cL$. The dual problem to \eqref{eq:mainmax} is
\begin{equation}\label{eq:dual}
{\rm minimize}\;\; F(\Sigma)-\<A_0,\Sigma\>\qquad \mbox{subject to}\qquad \Sigma-S_n\in \cL^\perp, \;\Sigma\in \S^m_+.   
\end{equation}
\end{prop}
\begin{proof}
This is a convex problem over a relatively open set $(S_n+\cL^\perp)\cap \S^m_+$. The optimum, if it exists, must be a stationary point: $\widehat\Sigma\in \S^m_+$, $\widehat \Sigma-S_n\in \cL^\perp$ and $\nabla F(\widehat \Sigma)-A_0\in (\cL^\perp)^\perp$. Note that  $\cL^\perp	$ is the linear space orthogonal to the affine space $\cL$ and so $(\cL^\perp)^\perp$ is the linear space parallel to $\cL$. Since $A_0\in \cL$ we recover the condition $\widehat L=\nabla F(\widehat \Sigma)\in \cL\cap \L^m_+$. This is exactly the same as \eqref{eq:KKTL}, which proves that the problem in \eqref{eq:dual} is equivalent to the problem in \eqref{eq:mainmax}.
\end{proof}

\subsection{The advantageous essentially smooth case}

The geometry of Bregman divergence optimization is discussed in detail by \cite{bauschke1997legendre}. We review now some of this theory to study the problem of existence of the Bregman estimator. These results will be used later in Section~\ref{sec:mixed} in our study of mixed parametrizations. 	We start with some standard definitions. 
	\begin{defn}
		A function $F\in \cconv(\S^m)$ with ${\rm int}(\dom(F))\neq \emptyset$ is called essentially smooth if it is differentiable on ${\rm int}(\dom(F))$ and $\|\nabla F(\Sigma_k)\|_F\to \infty$ for any sequence $(\Sigma_k)$ in ${\rm int}(\dom(F))$ that converges to the boundary of $\dom(F)$. 
	\end{defn}
	In our leading examples, the functions $F_A$, $F_C$, and $F_D$ are essentially smooth but $F_B$ is not because $\nabla F_B(\Sigma)=\Sigma$ is well defined also on the boundary of $\S^m_+$. In example (C) essential smoothness is quite surprising because the function $F_C$ itself extends to $\overline\S^m_+$ (see Appendix~\ref{sec:runex}).
	
\begin{defn}
A function $F\in \cconv(\S^m)$ is called Legendre if: (i) $F$ is essentially smooth, (ii) $F$ is strictly convex and differentiable on ${\rm int}(\dom(F))$. \end{defn}

Note that if $F\in \cE^m$ then condition (ii) is automatically satisfied.  Legendre functions are particularly well behaved for our purposes. 
\begin{prop}\label{prop:legendre}
	A convex function $F$ is Legendre if and only if  $F^*$ is Legendre.  In this case, the gradient mapping $\nabla F$ is an isomorphism between ${\rm int}(\dom(F))$ and ${\rm int}(\dom(F))$. In particular, if $F\in \cE^m$ is essentially smooth, then it is Legendre and  	
	 ${\rm int}(\dom(F^*))=\L^m_+$.
\end{prop}
\begin{proof}The first part follows from Theorem~26.5 in \cite{rockafellar1970convex}. For the second part, note that if $F\in \cE^m$ then ${\rm int}(\dom(F))=\S^m_+$ and $\L^m_+=\nabla F(\S^m_+)$ by definition. The second part follows then from the first part.	
\end{proof}

Importance of Proposition~\ref{prop:legendre} will be illustrated in this and the following sections. We start with a basic result on existence of the Bregman estimator. For a given affine subspace $\cL\subseteq \S^m$ define by $(\S^m_+)_\cL$ the image of the orthogonal projection of $\S^m_+$ onto $\cL$. Similarly, for a given $S\in \S^m$ denote by $S_\cL$ the orthogonal projection of $S$ on $\cL$.
\begin{thm}\label{th:optexist}
	Suppose $F\in \cE^m$ is essentially smooth. Then the optimum of $\g_n(L)$ over $L\in \cL\cup \L^m_+$ exists if and only if $(S_n)_\cL$ lies in  $(\S^m_+)_\cL$.
\end{thm}\begin{proof}By standard linear algebra, every $S\in \S^m$ can be uniquely decomposed as the sum $S_\cL+S_{\cL^\perp}$, where $S_\cL$ is the orthogonal projection of $S$ onto $\cL$ and $S_{\cL^\perp}$ lies in the orthogonal space $\cL^\perp$. Thus, we have that $(S_n)_\cL\in (\S^m_+)_\cL$ if and only if there exists $\widehat \Sigma\in \S^m_+$ such that $\<S_n,L\>=\<\widehat\Sigma,L\>$ for all $L\in \cL$. In this case $\g_n(L)=\<\widehat\Sigma,L\>-F^*(L)$ when restricted to $\cL$. Since $\widehat\Sigma\in \S^m_+={\rm int}(\dom(F))$, by Fact 2.11 in \cite{bauschke1997legendre}, each sub-level set $\{L\in \S^m: \g_n(L)\geq t\}$ is bounded for every $t\in \R$. Then it follows that  $\{L\in \cL: \g_n(L)\geq t\}$ is compact for every $t\in \R$, nonempty for some $t$, and so the maximum exists in $\cL\cap \dom(F^*)$. Since $F$ is essentially smooth, by Proposition~\ref{prop:legendre}, so is $F^*$. In particular, the optimum must be attained in ${\rm int}(\dom(F^*))=\L^m_+$.

To prove the other implication, suppose $(S_n)_\cL\notin(\S^m_+)_\cL$ but $\g_n(L)$ is optimized for some $\widehat L\in \cL\cap \L^m_+$. Denote $\widehat \Sigma=\nabla F^*(\widehat L)\in \S^m_+$. Since $F^*$ is essentially smooth, $\widehat L\in {\rm int}(\dom(F^*))=\L^m_+$ (c.f. Proposition~\ref{prop:legendre}) and the KKT conditions in Theorem~\ref{th:KKTg} assure that $\widehat \Sigma-S_n\in \cL^\perp$, which implies that $(\widehat \Sigma)_\cL=(S_n)_\cL$. This contradicts the assumption  that $(S_n)_\cL\notin(\S^m_+)_\cL$.
 \end{proof}


Recall that examples (A), (C), and (D) are all essentially smooth.  For a quick illustration why essential smoothness  is necessary consider the following example. 
\begin{ex}\label{ex:B13}
Let $m=3$ and consider the function $F_B$ in \eqref{eq:B} with $\nabla F_B(\Sigma)=\Sigma$. Suppose $\cL$ is given by a single equation $L_{13}=0$. Consider two matrices
	$$
	S_n=\begin{bmatrix}
		1 & \tfrac23& \tfrac23\\[.1cm]
		\tfrac23 & 1 & \tfrac23\\[.1cm]
		 \tfrac23& \tfrac23&1
	\end{bmatrix}\qquad\mbox{and}\qquad \widehat \Sigma=\begin{bmatrix}
		1 & \tfrac23& 0\\[.1cm]
		\tfrac23 & 1 & \tfrac23\\[.1cm]
		0& \tfrac23&1
	\end{bmatrix}.
	$$
	Note that $S_n$ is positive definite, $\widehat L=\widehat\Sigma\in \cL$, $\widehat\Sigma-S_n\in \cL^\perp$. However, $\widehat \Sigma$ is not positive definite and so it cannot be a solution to  \eqref{eq:mainmax}. This does not contradict Theorem~\ref{th:optexist} because $F_B$ is not essentially smooth.
\end{ex}

Of course, if $S_n\in \S^m_+$, then $(S_n)_\cL\in (\S^m_+)_\cL$ for every linear subspace $\cL\subseteq \S^m$. If $S_n$ is not positive definite, verifying whether $(S_n)_\cL$ lies in $(\S^m_+)_\cL$ may be complicated.  Note however that this condition does not depend on $F$ as long as $F$ is essentially smooth and so the conditions for existence of the Bregman estimator are exactly the same as the conditions for existence the MLE in Gaussian linear concentration models. In the case of Gaussian graphical models this has been extensively studied \cite{buhl1993existence,uhler2010,gross:sullivant:18,blekherman2019maximum,bernstein2020typical,bernstein2021maximum}. See Theorem~2.1 in \cite{uhler2010} for the special case of Theorem~\ref{th:optexist} for Gaussian graphical models and \cite{sturmfels2010multivariate} for the general linear Gaussian concentration model case.

In Section~\ref{sec:numeric} we discuss how to numerically solve problems \eqref{eq:mainmax} and \eqref{eq:dual}. But first we provide some statistical analysis of the underlying estimator.

\section{Basic statistical analysis}\label{sec:stat}

In this section we fix a parametrization of the affine space $\cL$ defining the entropic covariance model:
$$
L(\theta)\;=\;A_0+\sum_{i=1}^d \theta_i A_i\qquad \theta=(\theta_1,\ldots,\theta_d)\in \R^d,
$$
where $A_0,A_1,\ldots,A_d$ are fixed matrices in $\S^m$. For simplicity, we assume that the generators $A_i$ of $\cL$ form an orthonormal basis, that is, $\<A_i,A_j\>=0$ and $\|A_i\|_F=1$ for $i\neq j$. We also assume $A_0=0$, that is, $\cL$ is a linear subspace. Although no results in this section depend on this assumption, it significantly simplifies the notation. We also write $\g_n(\theta)$ for $\g_n(L(\theta))$. 

\subsection{Consistency and asymptotic Gaussianity}\label{sec:asym}
Consider a random sample $X_1,\ldots,X_n$ from a zero mean distribution with positive definite covariance matrix $\Sigma_0=\nabla F^*(L_0)$ with $L_0=L(\theta_0)$ for some $\theta_0\in \R^d$. The estimator obtained by solving \eqref{eq:mainmax} is a \emph{convex} M-estimator and the standard asymptotic theory, as presented in \cite{haberman1989concavity,niemiro1992asymptotics}, can be applied. Indeed, define $m:\R^d\times \R^m\to \R\cup\{+\infty\}$ by
$$
m(\theta,x)\;=\;F^*(L(\theta))-x^\top L(\theta) x
$$
then maximizing $\g_n(\theta)$ is equivalent to minimizing the strictly convex function
\begin{equation}\label{eq:Mn}
M_n(\theta)\;=\;\frac1n\sum_{i=1}^n m(\theta,X_i)\;=\;F^*(L(\theta))-\<L(\theta),S_n\>\;=\;-\g_n(\theta).	
\end{equation}
The corresponding minimizer $\widehat\theta_n$ is still called the Bregman estimator for $M_F(\cL)$ although now the estimator depends on the choice of basis $A_0,A_1,\ldots,A_d$.

Note that $\E S_n=\Sigma_0$ and the function
\begin{equation}\label{eq:M}
M(\theta)\;:=\;\E M_n(\theta)\;=\;F^*(L(\theta))-\tr(L(\theta)\Sigma_0)	
\end{equation}
is strictly convex in the interior of its domain. We have
$$
\dom(M)\;=\;\{\theta\in\R^d:L(\theta)\in \dom(F^*)\}\;\supseteq \;\{\theta:L(\theta)\in \L^m_+\}\;=:\;\Theta_+
$$
and we assume $\theta_0\in \Theta_+$. Note that if $\L^m_+$ is open, so is $\Theta_+$. Theorem~1 in \cite{niemiro1992asymptotics} immediately gives the following result.
\begin{prop}
	Suppose $F\in \cE^m$ and $\E(S_n)=\Sigma_0\in M_F(\cL)$. The Bregman estimator $\widehat\theta_n$ in $M_F(\cL)$ is a consistent estimator of $\theta_0$, where $\theta_0$ is the unique point such that $\nabla F^*(L(\theta_0))=\Sigma_0$.
\end{prop}
The main reason, why such nice results exist follows from the fundamental property of convex functions that pointwise convergence implies uniform convergence; see also Theorem~II.1  in \cite{andersen1982cox}. In consequence, the statistical analysis is much simpler than for general M-estimators \cite{geyer1994asymptotics}.

If we assume existence of higher order moments, we obtain a stronger conclusion. Let 
\begin{equation}\label{eq:h}
h(\theta,x)=\nabla_\theta m(\theta,x)=[\<\Sigma(\theta)-xx^\top,A_i\>]_{i=1}^d.	
\end{equation}
\begin{prop}\label{prop:niemiro2}
	If the distribution of $X$ has finite moments up to order $2r$ (equiv. $\E\|h(\theta,X)\|^r<\infty$) for some $r\geq 1$ then for every $\epsilon>0$
	$$
	\P(\sup_{k\geq n}\|\widehat\theta_k-\theta_0\|>\epsilon)\;=\;o(n^{1-r}), \quad n\to \infty.
	$$ 
\end{prop}
The proof follows immediately from Theorem~2 in \cite{niemiro1992asymptotics}.

We now turn to proving asymptotic Gaussianity. Denoting $\cS=\cov(X_1X_1^\top)$ to be the covariance of the matrix $X_1X_1^\top$, we get
\begin{equation}\label{eq:cS}
\cov(S_n)\;=\;\frac{1}{n^2}\sum_{i=1}^n \cov(X_iX_i^\top)\;=\;\frac1n \cS.	
\end{equation}
Note that $\cS$  is a covariance of a matrix valued random variable. Thus, similarly as in the standard vector-valued case, $\cS$ is a positive semidefinite and self-adjoint linear map from  $\S^m$ to $\S^m$  so that, for all $A,B\in \S^m$, we have
$$
\<A,\cS A\>\geq 0\quad\mbox{and}\quad \<A,\cS B \>=\<\cS A,B\>.
$$
The notation $\cS A$ denotes the action of the linear mapping $\cS:\S^m\to \S^m$ on $A$; see, for example, Section~2 and Appendix~A in \cite{lauritzen2023fundamentals}. Equivalently, it defines a bilinear mapping, for all $A,B\in \S^m$,
$$
\cS [A,B]\;=\;\E\left(\<X_1X_1^\top-\Sigma_0,A\>\cdot \<X_1X_1^\top-\Sigma_0,B\>\right).
$$

By \eqref{eq:nablag}, we can show that the Hessian $\nabla^2 M_n(\theta)$ of $M_n(\theta)$ does not depend on the data and it is equal to the Hessian of $M(\theta)$. We use the notation 
\begin{equation}\label{eq:hessMn}
\cI(\theta)\;=\;\nabla^2 M_n(\theta)\;=\;\nabla^2 M(\theta)\;\in\;\S^d.	
\end{equation}
Since $M(\theta)$ is strictly concave in $\Theta_+$, for every fixed $\theta\in \Theta_+$, $\cI(\theta)$ is a positive definite matrix. We write $\cI_0$  for $\cI(\theta_0)$. 

\begin{thm}\label{th:clasas}
Denote by $\widehat \theta_n$ the Bregman estimator obtained under data $S_n$ generated from a mean zero distribution with positive definite covariance $\E S_n=\Sigma_0=\nabla F^*(L(\theta_0))$ and such that $\V(S_n)=\tfrac1n \cS$ (the fourth order moments are assumed finite). Then 
$$\sqrt{n}(\widehat \theta_n-\theta_0)\;\overset{d}{\longrightarrow} N_m(0,\cI_0^{-1}\Omega \cI_0^{-1}),$$
where 
$\Omega_{ij}\;=\;\cS[A_i,A_j]$ for all $i,j=1,\ldots,m$.
\end{thm}
\noindent The proof of this and the remaining results of this section is given in Appendix~\ref{app:proofsstats}

\subsection{Finite sample bounds}

Obtaining finite sample bounds in specific distributional settings is also rather straightforward under additional strong convexity assumptions, but we need to be more careful about the existence of the estimator. In this section, we always assume that the solution to \eqref{eq:mainmax} exists. For example, by  Theorem~\ref{th:optexist}, the optimum exists with probability one if $n\geq m$, $F$ is essentially smooth.

We define centred versions of $M$ and $M_n$:
$$
\overline M(\omega)=M(\theta_0+\omega)-M(\theta_0),\qquad \overline M_n(\omega)=M_n(\theta_0+\omega)-M_n(\theta_0).
$$
We collect some basic facts about $\overline M(\omega)$, which follow immediately by the chain rule.
\begin{lem}\label{lem:gomega}
The function $\overline M(\omega)=F^*(L_0+L(\omega))-F^*(L_0)-\<\Sigma_0,L(\omega)\>$ is strictly convex in the interior of its domain and nonnegative. If $F^*$ is twice continuously differentiable then 
 $$
 \frac{\partial }{\partial \omega_i}\overline M(\omega)\;=\;\<\Sigma(\theta_0+\omega)-\Sigma_0,A_i\>,\qquad i=1,\ldots,d.
 $$	
 In particular, $\overline M(0)=0$ and $\nabla\overline M(0)=0$. Moreover,
 $$
 \frac{\partial^2 }{\partial \omega_i\partial \omega_j}\overline M(\omega)\;=\;\nabla^2 F^*(L_0+L(\omega))[A_i,A_j],\qquad i,j=1,\ldots,d
 $$	
 where $\nabla^2 F^*(L):\S^m\to \S^m$ for $\theta\in \Theta_+$ is a self-adjoint positive definite linear operator representing the second derivative of $F^*$ at $L$. \end{lem}
%

 Define $\Delta^{(n)}\in \R^d$ on each coordinate by
\begin{equation}\label{eq:Delta}
\Delta_i^{(n)}=\<S_n-\Sigma_0,A_i\> \qquad\mbox{for } i=1,\ldots,d.	
\end{equation}
\begin{lem}\label{lem:finite}
Fix $\epsilon>0$ and suppose that $\overline M$ is $\mu$-strongly convex in the $\epsilon$-ball, that is, $\overline M(\omega)-\tfrac{\mu}{2}\|\omega\|^2$ is convex on the $\epsilon$-ball. If $\widehat\theta_n$ is the Bregman estimator based on the sample covariance matrix $S_n$ from the true distribution with covariance corresponding to parameter $\theta_0$ then 
	$$
	\P(\|\widehat\theta_n-\theta_0\|\leq \epsilon)\;\;\geq \;\;\P(\|\Delta^{(n)}\|_\infty\leq  \tfrac{\mu\epsilon}{2\sqrt{d}}).
	$$
\end{lem}



By Lemma~\ref{lem:finite}, it is now enough to find probabilistic bounds on $\|\Delta^{(n)}\|_\infty$. If the distribution of the sample has a bounded support, finite sample bounds for $\|\Delta^{(n)}\|_\infty$ can be easily obtained using maximal inequalities for sub-Gaussian random variables. However, this strategy is not applicable in the Gaussian case. The following method is an alternative approach that has the potential for extension to other situations.

\begin{thm}\label{th:finite}
	Fix  $F\in \cE^m$, a linear subspace $\cL$ generated by an orthonormal basis $A_1,\ldots,A_d$. Suppose $X_1,\ldots,X_n$ is a random sample from $N(0,\Sigma_0)$ with $\nabla F(\Sigma_0)\in \cL\cap \L^m_+$. For $\epsilon>0$, let $\mu$ be such that $\overline M$ be $\mu$-strongly convex in the $\epsilon$-ball. If $\widehat \theta_n$ is the corresponding Bregman estimator of the true parameter $\theta_0$ then 
	$$
	\P(\|\widehat\theta_n-\theta_0\|>\epsilon) \;\leq\; \begin{cases}
		2d\exp\{-\frac{\mu^2\epsilon^2n}{32 d \|\Sigma_0\|^2}\} & \mbox{if }\epsilon\leq\tfrac{2\|\Sigma_0\|^2\sqrt{d}}{\mu},\\[2mm]
		2d\exp\{-\frac{\mu\epsilon n}{16\sqrt{d}\|\Sigma_0\| }\} & \mbox{otherwise}.
	\end{cases}
	$$
	In particular, for any $\delta\in (0,1)$, if $n\geq 8\log(\tfrac{2d}{\delta})$ then $\|\widehat\theta_n-\theta_0\|\leq \tfrac{4\|\Sigma_0\|}{\mu}\sqrt{\tfrac{2d}{n}\log(\tfrac{2d}{\delta})}$ with probability at least $1-\delta$.
\end{thm}

\subsection{Finite sample bounds for spectral sums}\label{app:finite}

Theorem~\ref{th:finite} shows that concentration of the Bregman estimator $\widehat\theta_n$ around the true value $\theta_0$ is driven by $d,n,\|\Sigma_0\|$, and $\mu$. We will now show how to further bound $\mu$ for our specific examples. 

\begin{prop}\label{prop:kappabound} Suppose $A_1,\ldots,A_d$ forms an orthonormal basis of the linear space $\cL\subseteq \S^m$.
	Suppose $F\in \cE^m$ is of the form $F(\Sigma)=\tr(\phi(\Sigma))$ with $\phi$ twice continuously differentiable. If there is an $\ell>0$ such that 
	\begin{equation}\label{eq:techcond}
\frac{\phi'(a)-\phi'(b)}{a-b}\;\leq\;\frac{\ell}{2}(\phi''(a)+\phi''(b)) \qquad\mbox{for all }a\geq b>0.	
\end{equation}
Then $\overline M$ is $\tfrac{1}{\alpha\ell}$- strongly convex in the $\epsilon$-ball, where $\alpha=\max_{\|\omega\|\leq \epsilon}\|\phi''(\Sigma(\theta_0+\omega))\|$.
\end{prop}

Proposition~\ref{prop:kappabound} reduced everything to finding $\alpha$ and $\ell$. We now briefly present how this result can be used in our three leading examples. \\[-1mm]

\noindent\textit{(A) Linear covariance models.} In this case, $F_B(\Sigma)=\tr(\phi(\Sigma))$ with $\phi(x)=x^2/2$ for $x> 0$, and  condition \eqref{eq:techcond} holds with $\ell=1$. Moreover, $\phi''(x)=1$ and so $\alpha=\|I_m\|=1$. By Proposition~\ref{prop:kappabound}, for every $\epsilon>0$, $\overline M$ is $\mu$-strongly convex in the $\epsilon$-ball with $\mu=1$. By Theorem~\ref{th:finite}, for any $\delta\in (0,1)$, if $n\geq 8\log(\tfrac{2d}{\delta})$ then
$$
\|\widehat\theta_n-\theta_0\|\;\leq\;4\|\Sigma_0\|\sqrt{\frac{2d}{n}\log\left(\frac{2d}{\delta}\right)}.
$$

	
\medskip

\noindent\textit{(B) Inverse covariance models.} In this case $F(\Sigma)=\tr(\phi(\Sigma))$ with $\phi(x)=-\log(x)$ for $x>0$ and $L=-\Sigma^{-1}$. As for linear covariance models, condition \eqref{eq:techcond} holds with $\ell=1$. Since $\phi''(x)=1/x^2$, $\phi''(\Sigma)=\Sigma^{-2}=L^2$ and we need a bound on $\alpha=\max_{\|\omega\|\leq \epsilon} \|L^{2}(\theta_0+\omega)\|$. We have
$$
\|L^2(\theta_0+\omega)\|\;=\;\|L(\theta_0+\omega)\|^2\;=\;\|L_0+L(\omega)\|^2\;\leq\;(\|L_0\|+\|L(\omega)\|)^2.
$$
We also have that 
$\max_{\|\omega\|\leq \epsilon}\|L(\omega)\|\leq \max_{\|\omega\|\leq \epsilon}\|L(\omega)\|_F\leq\epsilon$
and thus $\alpha\leq (\|L_0\|+\epsilon)^2$. It follows that $\overline M$ is $\mu$-strongly convex in the $\epsilon$-ball with 
$$
\mu\;=\;\frac{1}{\alpha}\;\geq\;\frac{1}{(\|L_0\|+\epsilon)^2}\;=\;\frac{\lambda_{\min}^2(\Sigma_0)}{(1+\epsilon\lambda_{\min}(\Sigma_0))^2}.
$$
By Theorem~\ref{th:finite}, for any $\delta\in (0,1)$, if $n\geq 8\log(\tfrac{2d}{\delta})$ then
$$
\|\widehat\theta_n-\theta_0\|\;\leq\;\frac{4\|\Sigma_0\|(1+\epsilon \lambda_{\min(\Sigma_0)})^2}{\lambda^2_{\min}(\Sigma_0)}\sqrt{\frac{2d}{n}\log\left(\frac{2d}{\delta}\right)}.
$$

\medskip
\noindent\textit{(C) Models linear in the matrix logarithm.} In this case $F(\Sigma)=\tr(\phi(\Sigma))$ with $\phi(x)=x\log(x)-x$ for $x>0$ and $L=\log(\Sigma)$. As for previous two cases, condition \eqref{eq:techcond} holds with $\ell=1$. Since $\phi''(x)=1/x$, $\phi''(\Sigma)=\Sigma^{-1}=e^{-L}$ and we get
\begin{equation}\label{aux:sigLlog}
\|e^{-L(\theta_0+\omega)}\|\;\leq \;e^{\|L_0+L(\omega)\|}\;\leq\;e^{\|L_0\|+\|L(\omega)\|}	
\end{equation}
and so $\alpha\leq e^{\|L_0\|+\epsilon}$. To explain the first inequality in \eqref{aux:sigLlog}, denote the eigenvalues of $L(\theta_0+\omega)$ by $\lambda_1\geq \lambda_2\geq \cdots\geq \lambda_m$. Then $\|e^{-L(\theta_0+\omega)}\|=e^{-\lambda_m}$. On the other hand, $\|L(\theta_0+\omega)\|=\|L(\theta)+L(\omega)\|$ is equal to $\max\{|\lambda_1|, |\lambda_m|\}$. The second inequality follows from the triangle inequality. It then follows that $\overline M$ is $\mu$-concave in the $\epsilon$-ball with 
$$
\mu\;\geq\;\frac{1}{e^{\|L_0\|+e}}\;=\;\frac{1}{e^{\epsilon}\|\Sigma_0\|}.
$$
By Theorem~\ref{th:finite}, for any $\delta\in (0,1)$, if $n\geq 8\log(\tfrac{2d}{\delta})$ then
$$
\|\widehat\theta_n-\theta_0\|\;\leq\;4\|\Sigma_0\|^2e^{\epsilon} \sqrt{\frac{2d}{n}\log\left(\frac{2d}{\delta}\right)}.
$$

%
%

\section{Mixed convex constraints}\label{sec:mixed}

It is evident that our analysis extends to models defined by arbitrary (closed) convex restrictions in $L=\nabla F(\Sigma)$, preserving the convex nature of the problem stated in equation \eqref{eq:mainmax}. However, in certain scenarios a portion of the restrictions may be easier expressed in $\Sigma$, while another part is better represented in $L$. We now explore how to handle such situations and present several theoretical implications. Our findings heavily rely on the geometric considerations underlying the mixed parametrization in exponential families \cite{barndorff:78}, as well as the study of the mixed convex exponential family setup of \cite{lauritzen2022locally}. Additionally, we provide new, more direct proofs of relevant geometric statements, aiming to make this theory more accessible.

\subsection{Mixed parametrization}

Consider a split of $\Sigma \in \S^m$  into two parts $\Sigma_A$ and $\Sigma_B$ and the corresponding split $L=(L_A,L_B)$. For instance, $\Sigma_A$ could consist of all the diagonal entries of $\Sigma$ and $\Sigma_B$ be the off-diagonal entries. In this section, we consider models that are given by convex restrictions on $\Sigma_A$ and $L_B$.  We start by discussing some motivating examples. 

\begin{ex}
	\cite{lauritzen2022locally} discussed a general family of such examples in the case when $L=-\Sigma^{-1}$. Their main motivation was to study Gaussian graphical models, given by zero restrictions on some entries of $L$, together with nonnegativity restrictions on the complementary entries of $\Sigma$. This framework leads to a natural notion of positive dependence for Gaussian graphical models. Additionally, a similar setting can be employed to test the equality of two Gaussian distributions within a given graphical model, as described in Example~3.6 in \cite{lauritzen2022locally}; see also \cite{djordjilovic2018searching}.
\end{ex}

\begin{ex}\label{ex:}
	Doubly Markov Gaussian models are defined by zero restrictions on $\Sigma$ and $L=-\Sigma^{-1}$; \cite{boege2023geometry}. Models of this form were first discussed by \cite{pearl1994can} in the context of Gaussian directed acyclic graph (DAG) models. In particular, $(\Sigma^{-1})_{ij}=0$ if $i$ and $j$ are not connected by an arrow and they have no common child in $G$. Similarly, $(\Sigma)_{ij}=0$ if $i$ and $j$ are not connected by an arrow and they have no common ancestor in $G$. This shows that zeros in $\Sigma$ and $\Sigma^{-1}$ carry information about the underlying graph. In Section~\ref{sec:other}, we briefly motivate in this context also models with zeros in higher powers of $\Sigma^{-1}$. 
\end{ex}

\begin{ex}
	Modelling correlations is used commonly in econometric GARCH models. In a recent paper, \cite{archakov2021new} considered a model in which $\Sigma_{ii}=1$ for all $i=1,\ldots,m$ and potential further linear restrictions are imposed on the off-diagonal entries of $L=\log\Sigma$. Another example of this kind, is to consider correlation matrices with zero restrictions on some off-diagonal entries of their inverse. 
\end{ex}

For a given subset of entries of a symmetric matrix, we consider the projection of the positive definite cone on those entries
$$
(\S^m_{+})_A\;:=\;\;\{\Sigma_A:\Sigma=(\Sigma_A,\Sigma_B)\in \S^m_+\}.
$$
We have a similar definition for $\L^m_+=\nabla F(\S^m_+)$
$$
(\L^m_{+})_B\;:=\;\;\{L_B:L=(L_A,L_B)\in \L^m_+\}.
$$
\begin{thm}\label{th:mixed}
	Let $F\in \cE^m$ be essentially smooth with $L=\nabla F(\Sigma)$. Then there is a one-to-one map between $\Sigma\in \S^m_+$ and  $(\Sigma_A,L_B)\in (\S^m_+)_A\times (\L^m_+)_B$.
\end{thm}
This result is essentially Theorem 5.34 in \cite{barndorff:78} but the assumption that $\dom(F)$ is open is replaced with the assumption that $F$ is essentially smooth. Our proof is also more difect than \cite{barndorff:78}. 
\begin{proof}
	Fix any $\overline \Sigma_A\in (\S^m_+)_A$ and any $\overline L_B\in (\L^m_+)_B$. Let $S\in \S^m_+$ be any element such that $S_A=\overline \Sigma_A$. Consider the problem: minimize $D(S,L)$ subject to the linear restrictions $L_B=\overline L_B$ on $L$. By Theorem~\ref{th:optexist}, this problem has a unique optimum $\widehat \Sigma$ with  $\widehat L=\nabla F(\widehat \Sigma)\in \L^m_+$ that satisfies the KKT conditions: $\widehat\Sigma_A=S_A=\overline \Sigma_A$ and $\widehat L_B=\overline L_B$. This shows that  every $(\overline \Sigma_A,\overline L_B)\in (\S^m_+)_A\times (\L^m_+)_B$ can be associated to a unique $\widehat \Sigma\in \S^m_+$. It is now enough to show that every $\widehat \Sigma\in \S^m_+$ is an image of some  $(\overline \Sigma_A,\overline L_B)\in (\S^m_+)_A\times (\L^m_+)_B$, but this is clear: take $\overline\Sigma_A=\widehat \Sigma_A$ and $\overline L_B=\widehat L_B$.
\end{proof}
%

Like in the standard exponential families, the most surprising part of this result is the following conclusion.
\begin{cor}[Variational independence]\label{cor:mixed}
	Assume conditions of Theorem~\ref{th:mixed} and let $S\in \S^m_+$, $\overline L\in \L^m_+$ be arbitrary. There exists a unique positive definite matrix $\widehat \Sigma$ that agrees with $S$ on $A$-entries and such that $\widehat L=\nabla F(\widehat \Sigma)$ agrees with $\overline L$ on the $B$-entries. 
\end{cor}

\begin{rem}
	Note that the proof of Theorem~\ref{th:mixed} gives the procedure to obtain the unique $\widehat \Sigma$ in Corollary~\ref{cor:mixed}. For the given $S\in \S^m_+$ and $\overline L\in \L^m_+$ minimize $D(S,L)$ subject to $L_B=\overline L_B$. Thus, $\widehat \Sigma$ can be found using any numerical procedure that solves this problem; see Section~\ref{sec:numeric}.
\end{rem}

For general $F\in \cE^m$ the conclusion of Theorem~\ref{th:mixed} does not hold. A counterexample can be recovered from Example~\ref{ex:B13}. Take $\Sigma_A=(\Sigma_{11},\Sigma_{22},\Sigma_{33},\Sigma_{12},\Sigma_{23})$ and $L_B=L_{13}$. Take $\Sigma$ to be $S_n$ from   Example~\ref{ex:B13} and let $L$ be the identity matrix. The corresponding $(\Sigma_A,L_B)$ is the matrix $\widehat\Sigma$ in  Example~\ref{ex:B13} but this one is \emph{not} positive definite. Perhaps it is also useful to see explicitly how this result works in a smaller example when Theorem~\ref{th:mixed} can be applied.
\begin{ex}
	Consider example (A) with $\nabla F_A(\Sigma)=-\Sigma^{-1}$. Let $m=2$ and consider the split $\Sigma_A=(\Sigma_{11},\Sigma_{22})$, $L_B=-(\Sigma^{-1})_{12}$. Here $(\S^2_+)_A=(0,\infty)^2$ and $(\S^2_+)_B=\R$. Suppose $\Sigma_{11}=\Sigma_{22}=1$ and denote $x=\Sigma_{12}$ and $y=(\Sigma^{-1})_{12}$. The claim is that $y$ can be chosen arbitrarily but here it follows clearly because $y=-x/(1-x^2)$ is a one-to-one mapping from $(-1,1)$ to $\R$.
\end{ex} 

It is interesting to note that we also have the equivalent of likelihood orthogonality; see Proposition~3.20 in \cite{sundberg}. In our setting, fixing the mixed parametrization $(\Sigma_A,L_B)$, the claim is that $\nabla_{\Sigma_A}\nabla_{L_B} D(S,L)$ is zero, namely, the Hessian of the Bregman divergence, when expressed in this parametrization, is block diagonal. This follows because $\nabla_{L_B} D(S_n,L)=\nabla_{L_B}F^*(L)-S_B=\Sigma_B-S_B$ and $\nabla_{\Sigma_A}(\Sigma_B-S_B)=0$ by variational independence.

\subsection{Unrestricted parametrization of correlation matrices}\label{sec:archakov1}

Motivated by temporal modelling of correlation matrices, \cite{archakov2021new} studied ways to map the set of correlation matrices in $\S_+^m$ into a Euclidean space $\R^{m(m-1)/2}$. It is not immediately obvious that such a mapping exists but the fact that $\log\Sigma$ maps $\S^m_+$ to $\S^m$ suggested a natural strategy. Their main result states that for any selection of the off-diagonal entries of $L\in \S^m$ there is a unique correlation matrix $R$ such that $L=\log R$.  

Our first observation is that this result is a special case of  Theorem~\ref{th:mixed}. Indeed, $F_C\in \cE^m$ is essentially smooth and so this theorem applies. Take $\Sigma_A$ to be the vector containing the diagonal entries of $\Sigma$ and $L_B$ to be the vector containing the off-diagonal entries of $L=\log(\Sigma)$. Then there is a one-to-one map from $\Sigma\in \S^m_+$ to $(\Sigma_A,L_B)\in (\S^m_+)_A\times (\L^m_+)_B$. In particular, we can fix $\Sigma_A\in (\S^m_+)_A$ to be the vector of ones and $L_B\in (\L^m_+)_B$ to be an arbitrary real vector to get Theorem~1 in  \cite{archakov2021new}. In this case $\L^m_+=\S^m$ and so
\begin{equation}\label{eq:ABcones}
(\S^m_+)_A\;=\;(0,\infty)^m\qquad\mbox{and}\qquad (\L^m_+)_B\;=\;\R^{m(m-1)/2}.	
\end{equation}

Note that the essential part of the above construction was not that $\L^m_+=\S^m$ (which motivated using the matrix logarithm) but that $(\L^m_+)_B=\R^{m(m-1)/2}$. This latter condition can be obtained for much simpler transformations. Indeed,  when $L = -\Sigma^{-1}$ as in Example~(A), it is always possible to choose the off-diagonal entries arbitrarily and adjust the diagonal entries to ensure the negative definiteness of $L$. Consequently, despite the fact that $\L_+^m=-\S_+^m$, we find that $(\L_+^m)_B = \mathbb{R}^{m(m-1)/2}$, which implies that the off-diagonal entries of the matrix inverse serves as an unconstrained parameterization for correlation matrices.
\begin{cor}
For any choice of the off-diagonal entries  there exists a unique correlation matrix $R$ whose inverse has precisely these off-diagonal entries. \end{cor}
A clear advantage of working with the inverse, rather than with the matrix logarithm, is that this is a well-understood algebraic map with efficient methods to compute it. Explicit numerical procedures are provided in \cite{amendola2021likelihood,llorens2022projected}. To see how it works,  let $\cL$ be the set of all $L\in \S^m$ with fixed off-diagonal entries $L_B$. Then $\cL^\perp$ is the set of matrices with zeros on the diagonal.  Take $A_0$ to be any matrix in $\cL$ and $S_n=I_m$. Solving the dual problem in  Proposition~\ref{cor:dual} with $F(\Sigma)=-\log\det(\Sigma)$ we obtain the unique $\widehat\Sigma$ with ones on the diagonal and such that $\widehat L\in \cL$.

\subsection{Estimation under mixed convex constraints}
Suppose now that we fix closed convex restrictions $ \cC_A$ on $\Sigma_A$ and closed convex restrictions $ \cC_B$ on $L_B$. Thus, the model is given by
\begin{equation}\label{eq:mixedm}
\{\Sigma\in \S^m:\; \Sigma_A\in \cC_A\cap (\S^m_+)_A\mbox{ and }L_B\in \cC_B\cap (\L^m_+)_B\}.	
\end{equation}
 We propose the following 2-step method to fit a model given by convex restriction on $\Sigma_A$ and convex restrictions on $L_B$.
\begin{enumerate}
	\item [(S1)] Minimize the Bregman divergence $D_F(S_n,L)=F(S_n)+F^*(L)-\<L,S_n\>$ subject to $L_B\in \cC_B\cap (\L^m_+)_B$. This is a convex optimization problem and denote the corresponding unique optimizer (if it exists) by $\widehat L$.
	\item [(S2)] Given $\widehat L$, minimize the Bregman divergence $D_F(\Sigma,\widehat L)=F(\Sigma)+F^*(\widehat L)-\<\widehat L,\Sigma\>$ subject to $\Sigma_A\in \cC_A\cap (\S^m_+)_A$. This is again a convex optimization problem and denote the corresponding minimizer (if it exists) by $\widecheck \Sigma$. 
\end{enumerate}
It is clear immediately by construction that $\widecheck \Sigma_A$ satisfies the given constraints on $\Sigma_A$. We also have the following, perhaps more surprising, result, which shows that $\widecheck\Sigma$ satisfies both the restrictions on $\Sigma_A$ and on $L_B$.
\begin{prop}\label{prop:2step}
Suppose $F\in \cE^m$ is essentially smooth. Let $\widehat L$ and $\widecheck \Sigma$ be as defined above and assume $\widehat L$ and $\widecheck L=\nabla F(\widecheck \Sigma)$ both lie in $\L^m_+$. Then  $\widecheck L_B=\widehat L_B\in \cC_B\cap (\L^m_+)_B$. In other words, $\widecheck\Sigma$ lies in the model \eqref{eq:mixedm}. 
\end{prop}
\begin{proof}
	If the optimum $\widecheck \Sigma$ in (S2) exists,  by convexity of the function and the underlying set $\cC_A\cap (\S^m_+)_A$, $\widecheck\Sigma$ must satisfy the KKT conditions
	$$
	\<\widecheck L-\widehat L,\Sigma-\widecheck \Sigma\>\geq 0\quad \mbox{for all }\Sigma\mbox{ s.t. }\;\;\Sigma_A\in \cC_A.
	$$
	Since $\S^m_+$ is open, a small perturbation $\widecheck \Sigma+T$ also lies in $\S^m_+$. Assuming that $T_A=0$ we can even conclude that $\widecheck \Sigma_A+T_A\in \cC_A\cap (\S^m_+)_A$. Since $T$ with $T_A=0$ is small but otherwise arbitrary, we conclude that $\widecheck L_B=\widehat L_B$. Moreover, $\widehat L_B\in \cC_B\cap (\L^m_+)$ because $\widehat L$ solves (S1). This concludes the proof.
\end{proof}
 
According to Proposition~\ref{prop:2step}, our procedure yields a point satisfying both types of constraints by solving two convex problems (S1) and (S2). Some statistical properties of the corresponding estimator can also be obtained following Section~\ref{sec:stat} above and Section~6 in \cite{lauritzen2022locally}. Let the Bregman estimator (BE) be the estimator $\widetilde \Sigma$ obtained by minimizing the Bregman divergence $D_F(S_n,\Sigma)$ subject to $\Sigma_A\in \cC_A$ and $L_B\in \cC_B$ (this is in general a non-convex optimization problem). In a way analogous to Theorem~6.1 in \cite{lauritzen2022locally} and with essentially the same proof, we expect that the estimations $\widetilde\Sigma_n$ and $\widecheck\Sigma_n$ are asymptotically equivalent, in the sense that $\sqrt{n}(\widetilde \Sigma_n-\widecheck \Sigma_n)=o_P(1)$. The importance of this comes from the fact that the Bregman estimator is an M-estimator and so its asymptotics under general restrictions is quite well understood \cite{geyer1994asymptotics}.

\section{Numerical optimization}\label{sec:numeric}

In this section we discuss numerical strategies to solve the problem \eqref{eq:mainmax} or the dual problem \eqref{eq:dual}.  In the specific example when $F(\Sigma)=-\log\det\Sigma$ there exist numerous approaches including coordinate descent and block-coordinate algorithms. Here we focus on general purpose solutions.

\subsection{Projected gradient descent}

In general, the second order information is hard to obtain for most choices of $F$ (c.f. \cite{lewis2001twice}). We thus first discuss a  first-order method. The simplest solution is to perform the projected gradient descent algorithm for the dual problem \eqref{eq:dual}: minimize $F(\Sigma)-\<A_0,\Sigma\>$ subject to $\Sigma-S_n\in \cL^\perp$. Note that if $\cL$ is a linear subspace, we can take $A_0=0$, which simplifies the formulas below. 

We initiate the algorithm at $\Sigma^{(0)}=S_n$, which is dually feasible at least as long as $S_n\in \S^m_+$. The gradient of $F(\Sigma)-\<A_0,\Sigma\>$ is $L-A_0$.  Denote by $\Pi_\cL^\perp(L-A_0)$ the orthogonal projection of $L-A_0\in \S^m$ to $\cL^\perp$. We move from $\Sigma^{(t)}$ to $\Sigma^{(t+1)}$ using the formula
$$
\Sigma^{(t+1)}\;=\;\Sigma^{(t)}-s_t \Pi_\cL^\perp(L^{(t)}-A_0).
$$
If $\Sigma^{(t)}$ is feasible, that is, if $\Sigma^{(t)}-S_n\in \cL^\perp$ then $\Sigma^{(t+1)}$ is also feasible and so the algorithm produces a sequence of feasible points. We can set the step size $s_t$ using backtracking, assuring in this way that the value of the function increases at each step. Since $F$ is strictly convex, this algorithm eventually converges to the optimum. This follows from the fact that the projected gradient descent is a special case of proximal gradient algorithms; see Section~10.4 in \cite{beck} for relevant results. 

Alternatively, it is possible to solve the primal problem \eqref{eq:mainmax}. We start by any feasible $L^{(0)}\in \cL\cap \L^m_+$. Then, at each iterate $L^{(t)}$, we project the gradient $-\Sigma^{(t)}+S_n$ onto $\cL$. We denote this projection by $\Pi_\cL$. Then we move
\begin{equation}
L^{(t+1)}\;=\;L^{(t)}-s_t \Pi_{\cL}(\Sigma^{(t)}-S_n).	
\end{equation}
Again, the step size can be chosen using backtracking. Note that here in some situations finding a feasible starting point $L_0$ may be problematic.

The main bottleneck in all these cases is that in each step we need to map between $\Sigma$ and $L$. This requires computing the spectral decomposition and so the complexity at each step is at least $O(m^3)$, which may be prohibitive if $m$ is very large. In the special case, when $F(\Sigma)=\tr(\phi(\Sigma))$ is a spectral sum, \cite{han2018stochastic} proposed to study the stochastic gradient descent based on stochastic truncation of the Chebyshev expansion of $F$ (or its conjugate). This and other techniques to approximate $F$ are discussed in Section~4.4 of \cite{higham}.

\subsection{Iterative Bregman projection}

Another possible approach is to employ an iterative projection algorithm as discussed by \cite{bauschke1997legendre,dhillon2008matrix},  which is structurally similar to iterative proportional scaling in exponential families. Observe that, by Theorem~\ref{th:KKTg}, the minimizer of $D_F(S,L)$ over the hyperplane $H$ given by $\<B,L\>=c$ must satisfy $\widehat\Sigma-S=\lambda B$ for some $\lambda\in \R$. Equivalently, we can solve the one-dimensional problem 
\begin{equation}\label{eq:1d}
\mbox{minimize}\quad F(S+\lambda B)-\lambda c\qquad \quad\mbox{such that }\;\lambda\in \R.	
\end{equation}
Suppose now that $\cL=\bigcap_{i=1}^k H_i$, where $H_i$ is a hyperplane $\<B_i,L\>=c_i$. We could iteratively ``project'' on $H_1,\ldots,H_k$ by minimizing the Bregman divergence to each $H_i$. Thus, we could run an iterative algorithm that starts with $\Sigma^{(0)}=S$ and, for $t\geq 0$,
$$L^{(t+1)}\;=\;\arg\min_{L\in H_i} D_F(\Sigma^{(t)},L),\qquad \Sigma^{(t+1)}=\nabla F^*(L^{(t+1)}),	
$$
or, equivalently,
\begin{equation}\label{eq:alg}
\lambda^{(t+1)}\;=\;\arg\min_{\lambda\in \R} F(\Sigma^{(t)}+\lambda B_i)-\lambda c_i,\qquad \Sigma^{(t+1)}=\Sigma^{(t)}+\lambda^{(t+1)} B_i,	
\end{equation}
where $i$ cycles around $\{1,\ldots,k\}$ in some, potentially random, order. The following result justifies this algorithm.
\begin{prop}[Theorem~8.1, \cite{bauschke1997legendre}]
	Suppose that $F\in \cE^m$ is essentially smooth and $S\in \S^m_+$. If $\dom(F)=\S^m_+$  then the algorithm in \eqref{eq:alg} converges to the global optimum of \eqref{eq:mainmax}. 
\end{prop}

\begin{ex}
	For illustration, we show how this could be used in the case when $F(\Sigma)=-\log\det(\Sigma)+\tfrac12\tr(\Sigma^2)$, which we motivate later in Proposition~\ref{prop:E}. Fix a graph $G$ with nodes $\{1,\ldots,m\}$ and consider the linear space
$$
\cL_G=\{L\in \S^m:\; L_{ij}=0\mbox{ if }i\neq j\mbox{ and } ij
\notin E\}.
$$
In this case the hyperplanes on which we project are defined by $c=0$ and $B=e_ie_j^\top +e_je_i^\top$ for $i\neq j$ and $ij\notin G$. In the $t$-th iteration we try to minimize $F(\Sigma^{(t)}+\lambda (e_ie_j^\top +e_je_i^\top))$ with respect to $\lambda\in \R$. Let $A=\{i,j\}$ and $C=\{1,\ldots,m\}\setminus \{i,j\}$. We write $\Sigma_{A|C}:=\Sigma_{A,A}-\Sigma_{A,C}\Sigma_{C,C}^{-1}\Sigma_{C,A}$. It is useful to observe that standard Schur complement arguments give that 
$$
\det (\Sigma^{(t)}+\lambda (e_ie_j^\top +e_je_i^\top))\;=\;\det(\Sigma^{(t)}_{C,C})\cdot \det\left(\Sigma^{(t)}_{A|C}+\begin{bmatrix}
	0 & \lambda\\
	\lambda & 0
\end{bmatrix}\right).
$$
Moreover, 
$$
\tfrac{1}{2}\tr((\Sigma^{(t)}+\lambda (e_ie_j^\top +e_je_i^\top))^2)=\tfrac12\tr((\Sigma^{(t)})^2)+2\lambda \Sigma_{ij}^{(t)}+\lambda^2.
$$
Denote $W=\Sigma^{(t)}_{A|C}\in \S^2$. Then
$$
\frac{{\rm d}}{{\rm d}\lambda}F(\Sigma^{(t)}+\lambda(e_ie_j^\top +e_je_i^\top))\;=\;\frac{2\lambda+2W_{12}}{\det(W)-2\lambda W_{12}-\lambda^2}+2\Sigma_{ij}^{(t)}+2\lambda.
$$
Equating this to zero results in a cubic polynomial equation, which can be solved exactly. By convexity of the problem, there may be only one real solution $\hat\lambda$ that leads to positive definite $\Sigma^{(t+1)}$. The condition is simple to check: $W_{11}W_{22}>(W_{12}+\hat \lambda)^2$.
\end{ex}

\section{Sparsity and positive definite completion}\label{sec:sparse}

 One important concept that has been extensively studied in high-dimensional statistics is sparsity \cite{hastie2015statistical}. In the context of covariance matrix estimation, Gaussian graphical models have proven to be particularly successful \cite{lau96}. It is customary to encode the sparsity pattern by a graph $G$ with $m$ nodes and edge set $E$. We denote the corresponding linear subspace by $\cL_G$:
  $$
 \cL_G\;:=\;\{L\in \S^m:\;L_{ij}=0\mbox{ if }ij\notin E\}.
 $$
 Note that we can take $A_0=0$ in this case and $\cL_G^\perp$ is given by zero restrictions on the complementary entries of $L$.
  
 \begin{ex}[Gaussian graphical models]\label{ex:ggm}
The multivariate Gaussian distribution $N_m(0,\Sigma)$ forms an exponential family with canonical parameter $K=\Sigma^{-1}$. Given a sample $X_1,\ldots,X_n$ from this model, the sufficient statistics is $S_n=\tfrac1n\sum_{i=1}^n X_i X_i^\top$. Denote the entries of $S_n$ by $S_{ij}$. For a given graph $G$, the Gaussian graphical model is given by imposing zero restrictions on some off-diagonal entries of $K=\Sigma^{-1}$
	$$
	M(G)=\{\Sigma\in \S^m_+: (\Sigma^{-1})_{ij}=0\mbox{ for }ij\notin E\}.
	$$
\end{ex}

Gaussian graphical models have made a significant impact on multivariate statistics and are commonly used even for non-Gaussian data. The elegant SKEPTIC approach introduced by \cite{skeptic} allows to extend the Gaussian setting to Gaussian copulas with minimal loss of efficiency and no loss of interpretability. Gaussian graphical models are also routinely employed beyond this favourable scenario. In such cases, the Gaussian log-likelihood is considered as a suitable loss function, and the zero restrictions correspond to conditional independence assumptions, albeit under the assumption of linear conditional independence. Interestingly, as demonstrated in \cite{rossell2020dependence}, some distributional settings, such as elliptical distributions, preserve certain non-linear conditional independence information when partial correlations vanish.

Zero restrictions on $\Sigma$ have also been explored in the literature, leading to the covariance graph model \cite{pearl1994can,kauermann1996dualization,chaudhuri2007estimation,drton2008graphical}. More recently, zero restrictions on $\log(\Sigma)$ have been considered in  \cite{battey17,battey2019sparsity,rybak2021sparsity}, with additional geometric motivations presented in \cite{pavlov2023logarithmically}. All of these models fall under the category of entropic models. We discuss yet another example.

\begin{ex}[Spatial autoregressive model (SAR)]\label{ex:sar}
This example is adapted from \cite{lesage2007matrix}.	Consider the spatial autoregression model: $Sy=S\beta+\varepsilon$, where the vector $y$ contains $m$ observations on the dependent variable, each associated with one region or point in space. The matrix $X$ represents $m\times k$ full column rank matrix of constants, which correspond to observations of $k$ independent variables for each region. We assume $\varepsilon\sim N_m(0,\sigma^2 I_m)$. The vector $\beta$ is the vector of parameters. In the SAR model, the matrix $S$ takes the form $S=I_m-\rho D$, where $D$ represents $m\times m$ nonnegative spatial weight matrix and $\rho$ reflects the magnitude of spatial dependence. If $i$-th unit interacts with the $j$-th unit in some meaningful way, they are called neighbours, which defined a graph $G$ with $m$ nodes. It is typically assumed that $D$ is symmetric and $D_{ij}=0$ for units that are not neighbours. Note that in this model the covariance matrix of $y$ takes the form $\Sigma=\sigma^2(I_m-\rho D)^{-2}$. As a result $\Sigma^{-1/2}\in \cL_G$. The corresponding entropic covariance model is generated by $F(\Sigma)=-2\tr(\sqrt{\Sigma})$ with $L=\nabla F(\Sigma)=-\Sigma^{-1/2}$; in Table~\ref{tab:1} we see that this is the dual construction to one of our running examples $F(\Sigma)=\tr(\Sigma^{-1})$.   
\end{ex}


 The next result follows from Corollary~\ref{cor:dual}.
 \begin{prop}\label{prop:sparsedual}
Let $S_n$ be the sample covariance matrix and consider the problem of maximizing $\g_n(L)=F^*(L)-\<S,L\>$ subject to $L\in \cL_G\cap \L^m_+$. 	The dual problem is to minimize $F(\Sigma)$ subject to $\Sigma_{ij}=S_{ij}$ for all $ij\in E$.
 \end{prop}

\begin{ex}\label{ex:L13}
    Suppose $m=3$ and suppose that $L_{13}=0$ is the only constraint defining $\cL_G$. Let $S\in \S^3_+$ be given by
    $$
    S=\begin{bmatrix}
    	4 & 1 &2\\
    	1 & 4 & 3\\
    	2 & 3 & 4
    \end{bmatrix}\quad\mbox{with}\qquad G=\overset{1}{\bullet}-\overset{2}{\bullet}-\overset{3}{\bullet}\quad\mbox{and}\qquad \widehat\Sigma=\begin{bmatrix}
    	4 & 1 &?\\
    	1 & 4 & 3\\
    	? & 3 & 4
    \end{bmatrix}.
    $$ By Proposition~\ref{prop:sparsedual}, irrespective of the form of $F$, the Bregman estimator $\widehat \Sigma$ is equal to $S$ on all the entries apart from the entries $(1,3)$ and $(3,1)$. The KKT conditions require that 
    $$
    (\nabla F(\widehat \Sigma))_{13}\;=\;0.
    $$
    We now show how this equation can be solved for our four running examples together with the new example $F_E(\Sigma)=F_A(\Sigma)+F_B(\Sigma)$ introduced later in Proposition~\ref{prop:E}. For $\nabla F_A=-\Sigma^{-1}$  we get 
$\widehat\Sigma_{13}\;=\;S_{12}S_{22}^{-1}S_{23}\;=\;\frac34$.     
For $\nabla F_B(\Sigma)=\Sigma$, $\widehat \Sigma_{13}=0$ as the resulting $\widehat \Sigma$ is positive definite (c.f. Example~\ref{ex:B13}).  If $F(\Sigma)$ is the negative von Neumann divergence, we need to rely on numerical computations developed in Section~\ref{sec:numeric} obtaining
$$
\widehat \Sigma\;=\;\begin{bmatrix}
    	4.0000 & 1.0000 & \mathbf{0.4298}\\
    	1.0000 & 4.0000 & 3.0000\\
    	\mathbf{0.4298} & 3.0000 & 4.0000
    \end{bmatrix}\qquad \widehat L\;=\;\log(\widehat \Sigma)\;=\;\begin{bmatrix}
    	1.3520 & 0.2721 & \mathbf{0.0000}\\
    	0.2721 &0.9305 &0.9806\\
    	\mathbf{0.0000} & 0.9806 & 0.9695
    \end{bmatrix}.
$$
For $\nabla F_D(\Sigma)=-\Sigma^{-2}$,  
$\widehat\Sigma_{13}\;=\;\frac13(64-\sqrt{3754})\;\approx\;0.91$. Finally, for $\nabla F_E(\Sigma)=\Sigma-\Sigma^{-1}$, $\widehat \Sigma_{13}\approx 0.105$.
\end{ex}

\section{Jordan algebras}\label{sec:jordan}

In this section we cover a set of particularly nice linear constraints with some history in covariance matrix estimation. This allows us to recognize various scattered results in a unifying way. We start with a definition; see \cite{jensen1988covariance}.
\begin{defn}
	A linear space $\cL\subseteq \S^m$ is called a Jordan algebra of symmetric matrices if 
	\begin{equation}\label{eq:jordan}
\forall A,B\in \cL \qquad \tfrac12(AB+BA)\in \cL.	
\end{equation}
\end{defn}
\noindent For a Jordan algebra $\cL$, if $A\in \cL$ then $A^2\in \cL$ (or $A^n\in \cL$ in general for $n\geq 1$). This condition is in fact equivalent to \eqref{eq:jordan} by the fact that $\cL$ is a linear subspace and by the identity
$AB+BA\;=\;(A+B)^2-A^2-B^2$.

To motivate Jordan algebras in statistics we provide the following examples. 
 \begin{ex}
	Consider the correlation model with an additional restriction that all off-diagonal entries are equal to each other. This is known as the equicorrelation model \cite{amendola2021likelihood}. In Proposition~2 of \cite{archakov2021new}, it was shown that the logarithm of an equicorrelation matrix has equal off-diagonal entries. This result is a special case of Proposition~\ref{prop:jordan} by observing that the set of matrices with equal diagonal entries and equal off-diagonal entries forms a Jordan algebra.\end{ex}
	
There is a natural way to generalize this example.

\begin{ex}
Let $\cS_m$ be the symmetric group on $\{1,\ldots,m\}$, which we identify with the set of permutation matrices in $\R^{m\times m}$. Fix a subgroup $\cG\subseteq \cS_m$ and consider the set 
$$\{\Sigma\in \S^m:\;U\Sigma U^\top =\Sigma\mbox{ for all }U\in \cG\}.$$   	
This set of restrictions has been studied as the RCOP model in  \cite{hojsgaard2008graphical}. In the case when $\cG=\cS_m$ we get matrices whose all diagonal elements and all off-diagonal elements are equal. 
\end{ex}

Another simple example, discussed by \cite{chiu1996matrix}, is when $\cL=\{\alpha I_m+\beta U: \alpha,\beta\in \R\}$ and $U$ is an idempotent matrix ($U^2=U$), or more generally, $U^2=\gamma U$ for some $\gamma\in \R$. 	\cite{szatrowski2004patterned} also gives an overview of interesting linear restrictions that correspond to Jordan algebras and motivates them statistically; see also \cite{rubin1982finding}. The following example, motivated by spatial modelling, is also interesting.
\begin{ex} Recall the spatial autoregressive model (SAR) presented in Example~\ref{ex:sar}. The main motivation in \cite{lesage2007matrix} was to come up with a version of this model, which is easier to handle statistically and computationally. Their starting observation was that in practice of SAR modelling the matrix $D$ is row normalized so that $D\mathbf 1_m=\mathbf 1_m$. Row-stochastic spatial weight matrices have a long history of application in spatial statistics; e.g.   \cite{ord1975estimation}. Consider the subspace
	$$
	\cL\;=\;\{L\in \S^m:\; \exists \alpha\in \R\mbox{ s.t. }L\bs 1_m=\alpha \bs 1_m	\}.
	$$
This $\cL$ satisfies conditions of Proposition~\ref{prop:jordan}. Note that $$(I_m-\rho D)^{-1}\;=\;I_m+\sum_{n\geq 1} \rho^n D^n.$$
In the MESS model of \cite{lesage2007matrix}, the matrix $S=(I_m-\rho D)$ in usual SAR formulation is replaced with the matrix exponential $\exp(-\alpha D)$, for $\alpha>0$, in which case 
$$
(\exp(-\alpha D))^{-1}\;=\;\exp(\alpha D)\;=\;I_m+\sum_{n\geq 1} \frac{\alpha^n}{n!}D^n,
$$
which corresponds to exponential decrease of influence for higher-order neighbors. Note that the MESS model  assumes $\log\Sigma\in \cL_G\cap \cL$. The fact that $\cL$ is a Jordan algebra, greatly simplifies computations.  
\end{ex}

Szatrowski showed in a series of papers \cite{szatrowski1978explicit,szatrowski1980necessary,szatrowski2004patterned}  that the MLE for linear Gaussian covariance models has an explicit representation, i.e., it is a known linear combination of entries of the sample covariance matrix, if and only if  $\cL$ forms a Jordan algebra. Furthermore, Szatrowski proved that for this restrictive model class the MLE is given in the closed form. We will generalize this result here by first generalizing the main result of \cite{jensen1988covariance}. 


\begin{prop}\label{prop:jordan}
	Suppose $F\in \cE^m$ takes the form $F(\Sigma)=\tr(\phi(\Sigma))$ with $\phi$ analytic on $(0,+\infty)$.  Suppose $\cL$ is a linear subspace which satisfies \eqref{eq:jordan} and  $I_m\in \cL$,  then $\Sigma\in \cL\cap \S^m_+$ if and only if  $L=\nabla F(\Sigma)\in \cL\cap \L^m_+$. 
\end{prop}
\begin{proof}
%
Since  $\phi$ is analytic on $(0,+\infty)$, $\phi'$ is also analytic and we can take its series expansion around 1:
$$
\phi'(x)=\sum_{n\geq 0}c_n (x-1)^n\qquad\mbox{for some }c_n\in \R, n\geq 0.
$$
In consequence, if $\Sigma\in \cL\cap \S^m_+$ and $\cL$ satisfies \eqref{eq:jordan} with $I_m\in \cL$ then $(\Sigma-I_m)^n\in \cL$ for all $n\geq 1$ and so
$$
\nabla F(\Sigma)\;=\;\phi'(\Sigma)\;=\;c_0I_m+\sum_{n\geq 1} c_n (\Sigma-I_m)^n\;\in \;\cL\cap \L^m_+.
$$  	
This shows the right implication. For the left implication we use the fact that $\phi''(x)>0$ (strict convexity) and so $\phi'$ is strictly increasing. The inverse of $\phi'$ is also analytic by the Lagrange Inversion Theorem (e.g. Theorem~5.4.2 of \cite{Stanley_2023}). Now we can apply exactly the same argument as above to the inverse of $\phi'$. 
\end{proof}

\begin{rem}
If $\cL$ forms a Jordan algebra and $I_m\in \cL$, then we also do not have to worry about the statistical intepretability of the linear restrictions on $L=\nabla F(\Sigma)$ because they are exactly equivalent to the same linear restrictions on $\Sigma$.	
\end{rem}

This is how Jordan algebra structure leads to trivial estimation procedures.
\begin{prop}
	Suppose that $F\in \cE^m$ and $\L^m_+=\nabla F(\S^m_+)$ is open. If $\cL$ is a Jordan algebra and $I_m\in \cL$, then the optimum in \eqref{eq:mainmax} is  given in a closed form: $\widehat L=\nabla F(\widehat \Sigma)$, where $\widehat \Sigma$ is the orthogonal projection of $S_n$ on $\cL$. 
\end{prop}
\begin{proof}
	By Theorem~\ref{th:KKTg}, the optimum in \eqref{eq:mainmax}, if it exists,  is uniquely given by the pair $(\widehat \Sigma,\widehat L)\in \S^m_+\times \L^m_+$ with $\widehat L=\nabla F(\widehat \Sigma)$ satisfying \eqref{eq:KKTL}: $\widehat L\in \cL$ and $\widehat \Sigma-S_n\in \cL^\perp$.	Since $\cL$ is a Jordan algebra and $I_m\in \cL$, by Proposition~\ref{prop:jordan}, equivalently $\widehat\Sigma\in \cL$. But the condition $\widehat \Sigma\in \cL$ and $\widehat \Sigma-S_n\in \cL^\perp$ says exactly that $\widehat\Sigma$ is an orthogonal projection of $S_n$ onto $\cL$. 
\end{proof}


\section{Discussion}\label{sec:other}

In this paper we presented a flexible family of models for covariance matrices together with a canonical estimation procedure. We presented that this setting brings new insights into the geometry of covariance matrices, which can be applied to design new statistical procedures. To conclude the paper, we briefly outline some of the questions that immediately arise for future work.  
\medskip

\noindent\textit{Unrestricted parametrizations of covariance matrices.} The main motivation for using the matrix logarithm transformation in \cite{chiu1996matrix} was that it maps $\S^m_+$ bijectively to $\S^m$. Using our geometric insights, we can easily provide a tractable alternative for this transformation; an algebraic map $\nabla F$ for which $\nabla F(\S^m_+)=\S^m$. 
\begin{prop}\label{prop:E}
	Consider the map $F(\Sigma)=-\log\det(\Sigma)+\tfrac\lambda2\tr(\Sigma^2)$ for any $\lambda>0$. Then $F\in \cE^m$ is essentially smooth and $\nabla F(\Sigma)=\lambda\Sigma-\Sigma^{-1}$ maps bijectively $\S^m_+$ to $\S^m$. 
\end{prop} 
\begin{proof}
The underlying function $\phi$ is $-\log(x)+\tfrac\lambda2 x^2$ when $x>0$ and $+\infty$ for all other $x$. It is strictly convex and differentiable in $(0,+\infty)$. Moreover, $\phi'(x)=\lambda x-1/x$ and so $|\phi'(x)|\to \infty$ as $x\to 0^+$. This shows that $F\in \cE^m$ is essentially smooth. As $\phi$ maps $(0,+\infty)$ to $\R$, the result follows. \end{proof}

\medskip

\noindent\textit{Natural exponential families for $S\in \S^m_+$:} As suggested by one of the referees, this set-up implicitly leads to definition of a natural exponential family; see Theorem~4 in \cite{banerjee2005clustering}. If $F\in \cE^m$ is essentially smooth then we could study the family of distributions over $S\in \S^m_+$ with density  $p_L(S)=h(S)\exp\{\<S,L\>-F^*(L)\}$. This could be particularly interesting in the case when both $h$ and $F^*$ are spectral sums.
\medskip

\noindent\textit{Gaussian Bayesian networks.} Consider a directed acyclic graph (DAG) $G$ whose nodes represent components of the random vector $X=(X_1,\ldots,X_m)$. We say that the distribution of $X$ lies in a Gaussian linear structural model over $G$ if
$$
X_i\;=\;\sum_{j\to i \in G} \lambda_{ij} X_j+\varepsilon_i,
$$
where $\lambda_{ij}\in \R$ and $\varepsilon_i$ is independent of $X_j$ for each parent $j$ of $i$ in $G$, and the $\varepsilon_i$'s are mutually independent. Denote by $\Lambda\in \R^{m\times m}$ the matrix with entries $\lambda_{ij}$ if $j\to i$ in $G$ and zero otherwise. Then the covariance matrix $\Sigma$ of $X$ satisfies
$\Sigma=(I-\Lambda)^{-1}\Omega (I-\Lambda)^{-\top}$,
where $\Omega$ is the diagonal covariance matrix of $\varepsilon$. Taking the inverse $K=\Sigma^{-1}$, we get  $K=L\Omega^{-1}L^\top$ (see e.g. Proposition~2.1 in \cite{trek}), where $L=(I-\Lambda)^\top$, and so $L_{ij}=0$ unless $i=j$ or $i\to j$ in the underlying DAG.

Looking at the second power of $\Sigma^{-1}$ we see that $$
(K^2)_{ij}=(L\Omega^{-1}L^\top L\Omega^{-1}L^\top)_{ij}=\sum_{u,v,w}\frac{L_{iu}L_{vu}L_{vw}L_{jw}}{\Omega_{uu}\Omega_{ww}}.
$$
In particular, $(K^2)_{ij}=0$ if the graph does not contain a structure $i\to u \leftarrow v \to w \leftarrow j$ in $G$ or any simpler structure obtained from this by contracting some of the arrows. We get a similar interpretation for higher powers of $\Sigma^{-1}$.
\medskip

\begin{acks}[Acknowledgments]
Special thanks go to the anounymous referees and the associate editor for many insightful suggestions. I would also like to thank Daniel Bernstein, Christian Brownlees, Hengchao Chen, Sean Dewar, Mathias Drton, Steven Gortler, G\'{a}bor Lugosi, Geert Mesters, Frank R\"{o}ttger, Alex Valencia, and  Jierui Zhu for helpful remarks. This research was supported by a grant from the Natural Sciences and Engineering Research Council of Canada (NSERC, RGPIN-2023-03481). 
\end{acks}

\appendix

\section{Spectral functions}\label{sec:spec}

\subsection{Basic convex analysis on $\R^m$}\label{sec:basicconv}

In this section we briefly review some basic and very streamlined results in convex analysis. {We refer to \cite{rockafellar1970convex,hiriart2012fundamentals} as good references for convex analysis. Chapter~5 of \cite{barndorff:78} provides an exposition of the most statistically-relevant results from \cite{rockafellar1970convex}. Our discussion of Bregman divergences is closely related to \cite{bauschke1997legendre}}

By $\conv(\R^m)$ denote the set of convex function $f\in \R^m\to \R\cup \{+\infty\}$ that are not identically equal to $+\infty$ (those are sometimes called proper convex functions). The domain of $f\in \conv(\R^m)$ is the non-empty set $\dom(f)=\{x\in \R^m: f(x)<+\infty\}$. By $\cconv(\R^m)$ denote the class of functions in $\conv(\R^m)$ that are lower semicontinuous on $\R^m$ - these are also known as closed convex functions. Recall that $f$ is lower semicontinuous if the lower level-set $\{x: f(x)\leq t\}$ is closed for all $t\in \R$.

Let $f:\R^m\to \R\cup\{+\infty\}$ be \emph{any} function that is not identically equal to $+\infty$. Suppose, in addition, that there is an affine function minorizing $f$ on $\R^m$.  If $f\in \conv(\R^m)$, this is automatically satisfied; see Proposition~B1.2.1 in \cite{hiriart2012fundamentals}. The conjugate function $f^*$ is defined for $s\in \R^m$ as
$$
f^*(s)\;:=\;\sup_{x\in \R^m}\{\<s,x\>-f(x)\}.
$$
 If $f$ is minorized by $\<s_0,x\>-b$, for some $s_0\in \R^m$ and $b\in \R$. Then $f^*(s_0)\leq b$ and $f^*(s)>-\infty$ for all $s$ because $\dom(f)\neq \emptyset$. Thus $f^*:\R^m\to \R\cup\{+\infty\}$ and it is not  equal to $+\infty$ everywhere. Since it is a pointwise supremum of linear functions, $f^*$ is convex. In fact, $f^*\in \cconv(\R^m)$ irrespective of whether $f$ is convex. We actually have the following fundamental result: $f\in \cconv(\R^m)$ if and only if $f=(f^*)^*$; Corollary E.1.3.6 in \cite{hiriart2012fundamentals}.

Note that, by definition, $f(x)+f^*(s)-\<s,x\>\geq 0$ for every $s,x$. The conjugate function can be used to conveniently define the subdifferential of $f$ at $x$ for any $f\in \conv(\R^m)$
$$
\partial f(x)\;:=\;\{s\in \R^m: f(x)+f^*(s)-\<s,x\>=0\};
$$
see Theorem~E1.4.1 in \cite{hiriart2012fundamentals}. Moreover,  $f$ is differentiable at $x$ with gradient $\nabla f(x)$ if and only if $\partial f(x)=\{\nabla f(x)\}$ and if $x\notin \dom(f)$ then $\partial f(x)=\emptyset$. 

\subsection{General results on spectral functions on $\S^m$}
Establishing convexity of a general function $F:\S^m\to \R\cup\{+\infty\}$ and computing its gradient may be complicated. We note however that our running examples are spectral functions \cite{lewis1996convex,watkins1974convex}. For reader's convenience we briefly mention relevant results and definitions. 
\begin{defn}
A function $F:\S^m\to \R\cup\{+\infty\}$ is a spectral function if $F(U^\top \Sigma U)=F(\Sigma)$ for all $\Sigma\in \S^m$ and orthogonal $U$. In particular, $F(\Sigma)$ depends on the eigenvalues of $\Sigma$ only.
\end{defn}
Associated with any spectral function is a symmetric real-valued function $f:\R^m\to \R\cup\{+\infty\}$. Specifically, we define $f(\lambda)=F({\rm diag}(\lambda))$, 
where ${\rm diag}(\lambda)$ is the diagonal matrix with $\lambda=(\lambda_1,\ldots,\lambda_m)$ on the diagonal. Denote by $\lambda(\Sigma)=(\lambda_1(\Sigma),\ldots, \lambda_m(\Sigma))$ the spectrum of $\Sigma$ in a decreasing order. The spectral functions are precisely those of the form $F(\Sigma)=F({\rm diag}(\lambda(\Sigma)))=f(\lambda(\Sigma))$. The following important result appears as Corollary~2.4 in \cite{lewis1996convex}; see also \cite{davis1957all}.
\begin{thm}\label{th:speccconv}
    A symmetric function $f$ on $\R^m$ satisfies $f\in \cconv(\R^m)$ if and only if the associated spectral function $F$ defined by $F(\Sigma)=f(\lambda(\Sigma))$ satisfies $F\in \cconv(\S^m)$. 
\end{thm}
\begin{proof}[Sketch proof] The left implication follows easily by restricting to the linear subspace of diagonal matrices $\Sigma$. So suppose $f\in \cconv(\R^m)$. To show that $F\in \cconv(\S^m)$ we show that $F=(F^*)^*$. For this we first show that if $F(\Sigma)=f(\lambda(\Sigma))$ for $f:\R^m\to \R$ symmetric then $F^*(L)=f^*(\lambda(L))$. For this, we use the von Neumann trace inequality which states that $\<L,\Sigma\>\leq \lambda(L)^\top \lambda(\Sigma)$ with equality if and only if there exists an orthogonal matrix $U$ such that $U^\top \Sigma U$ and $U^\top L U$ are simultaneously diagonal; see \cite{mirsky1959trace}. It follows that 
\begin{eqnarray*}
	F^*(L)&=&\sup_{\Sigma\in \S^m}\{\<L,\Sigma\>-F(\Sigma)\}\;\leq\;\sup_{\Sigma\in \S^m}\lambda(L)^\top \lambda(\Sigma)-f(\lambda(\Sigma))\\
	&=&	 \sup_{u\in \R^m}\{ \lambda(L)^\top u-f(u)\}\;=\;f^*(\lambda(L)).
\end{eqnarray*}
Note that since the von Neumann inequality can become equality as described above, the inequality above becomes equality proving that $F^*(L)=f^*(\lambda(L))$. Suppose now that $f\in \cconv(\R^m)$ or, equivalently, $f=(f^*)^*$. By the argument above applied to $F^*$ 
$$
(F^*)^*(\Sigma)\;=\;(f^*)^*(\lambda(\Sigma))\;=\;f(\lambda(\Sigma))\;=\;F(\Sigma).
$$\end{proof}

Computing the gradient of a spectral function  $F$ also relies on computing the gradient of $f$. The following result will be useful; see Corollary~3.2 \cite{lewis1996convex}.
\begin{thm}\label{th:Fder}
    Suppose that the function $f\in \cconv(\R^m)$ is  symmetric.  If $F(\Sigma)=f(\lambda(\Sigma))$ is the associated spectral function and $\Sigma=U{\rm diag}(\lambda(\Sigma)) U^\top$ for an orthogonal matrix $U$ then for every $\Sigma$ such that $f$ is differentiable at $\lambda(\Sigma)$, $F$ is differentiable at $\Sigma$ and
    $$
    \nabla F(\Sigma)\;=\; U\, {\rm diag}(\nabla f(\lambda(\Sigma))) \,U^\top.
    $$
\end{thm}
\begin{proof}[Proof Sketch]
In the proof of Theorem~\ref{th:speccconv} we showed that $F^*(L)=f^*(\lambda(L))$. By Section~\ref{sec:basicconv}, we have $L\in \partial F(\Sigma)$ exactly when
	$$
	\<L,\Sigma\>\;=\;F(\Sigma)+F^*(L)\;=\;f(\lambda(\Sigma))+f^*(\lambda(L))\;\leq\;\lambda^\top(L)\lambda(\Sigma)\;\leq\;\<L,\Sigma\>,
	$$
	where the first inequality follows from the definition of $f^*$ and the second inequality follows again by the von Neumann trace inequality \cite{mirsky1959trace}. The fact that the left most expression is equal to the right most expression implies that both inequalities are equalities and so: (i)  $\lambda(L)\in \partial f(\lambda(\Sigma))$, (ii) if $\Sigma=U{\rm diag}(\lambda(\Sigma))U^\top$ then $L=U{\rm diag}(\lambda(L))U^\top$. If $f$ is differentiable at $\lambda(\Sigma)$ then $\nabla f(\lambda(\Sigma))=\lambda(L)$ and $\nabla F(\Sigma)=L=U {\rm diag}(\lambda(L))U^\top$, which concludes the proof.
\end{proof}
 
We note that this result is true more generally for spectral functions that are not convex; see Theorem~1.1 in \cite{lewis1996derivatives}. The computation of the second order derivatives is generally much more complicated \cite{lewis2001twice}. 

It is worth noting that all our examples have a special form
\begin{equation}\label{eq:spsumf}
f(\lambda)\;=\;\sum_{i=1}^m \phi(\lambda_i)	
\end{equation}
for some smooth function $\phi:\R\to \R\cup\{+\infty\}$. For example, $-\log\det(\Sigma)=-\sum_{i=1}^m \log\lambda_i(\Sigma)$ and $\tfrac12\tr(\Sigma^2)=\tfrac12\sum_{i=1}^m \lambda_i^2(\Sigma)$. Using the matrix function notation (e.g. \cite{higham}) we can write it more elegantly as
$F(\Sigma)\;=\;\tr(\phi(\Sigma))$.
Such functions are also sometimes called spectral sums. The following follows from Theorem~\ref{th:Fder}.
\begin{prop}\label{prop:Fder}
	Suppose that $\phi$ is differentiable in $(0,+\infty)$ and $F(\Sigma)={\rm tr}(\phi(\Sigma))$. Then, for every $\Sigma\in \S^m_+$,  we get
$\nabla \tr(\phi(\Sigma))\;=\;\phi'(\Sigma)$.
\end{prop}
\begin{proof}
	If $f(\lambda)=\sum_i \phi(\lambda_i)$ then $\tfrac{\partial f}{\partial \lambda_i}(\lambda)=\phi'(\lambda_i)$. By Theorem~\ref{th:Fder}, if $\Sigma=U{\rm diag}(\lambda(\Sigma))U^\top$ then
	$$
	\nabla \tr(\phi(\Sigma))\;=\;U\phi'({\rm diag}(\lambda(\Sigma)))U^\top\;=\;\phi'(\Sigma).
	$$
\end{proof}

\subsection{Running examples}\label{sec:runex}\label{sec:conjugate}

Many results in this paper rely on the fact that the underlying function $F$ lies in $\cconv(\S^m)$. In this section we study specifically spectral sums.
\begin{lem}\label{lem:spectrlcconv}Suppose $F(\Sigma)=\tr(\phi(\Sigma))$ then $F^*(L)=\tr(\phi^*(L))$. Moreover, $F\in \cconv(\S^m)$ if and only if $\phi\in \cconv(\R)$.
\end{lem}
\begin{proof}
First note that in the proof of Theorem~\ref{th:speccconv} we showed $F^*(L)$ is induced by the conjugate of $f(\lambda)=\sum_i\phi(\lambda_i)$. We have $$
f^*(y)\;=\;\sup_{x\in \R^m} \{\<x,y\>-\sum_{i=1}^m \phi(x_i)\}\;=\;\sum_{i=1}^m \phi^*(y_i),
$$
which proves the first statement. This implies that $F^{**}(\Sigma)=\tr(\phi^{**}(\lambda))$ and in consequence also the second part of the lemma.
\end{proof}

As we said earlier, we normally define $F$ over $\S^m_+$ and then we extend it to $\S^m_+$ by taking the lower semicontinuous closure. We now explain how this works for our running examples. 

Suppose that $\phi$ is a proper convex function. Taking $\tr(\phi(\Sigma))+\i_{\overline{\S}^m_+}(\Sigma)$ in Remark~\ref{rem:addi} corresponds to adding to $\phi$ the indicator of the closed interval $[0,+\infty)$. Here the semicontinuous closure is easy to calculate directly. Concretely, the functions $\phi:\R\to \R\cup\{+\infty\}$ that we consider in this paper take the form
$$
\text{\bf (i)}\;\; \phi(x)=+\infty\;\;\;\mbox{for }x<0; \quad \text{\bf (ii)}\;\; \phi(x)<+\infty\;\;\;\mbox{for }x>0; \quad\text{\bf (iii)}\;\; \phi(0)=\lim_{x\to 0^+} \phi(x).
$$
We have $\dom(\phi)=(0,+\infty)$ or $\dom(\phi)=[0,+\infty)$ depending on whether the limit in (iii) is finite or not. 

Now we can present our running examples more formally. In example (A) our function was defined by $\phi_A(x)=-\log(x)$ for $x>0$. We have $\lim_{x\to 0^+} \phi_A(x)=+\infty$ and so:
\begin{equation}\label{eq:A}
	\phi_A(x)\;=\;\begin{cases}
		-\log(x) & x>0\\
		+\infty & x\leq 0
	\end{cases},\qquad F_A(\Sigma)\;=\;\begin{cases}
		-\log\det(\Sigma) & 
		\Sigma\in \S^m_+\\
		+\infty & \Sigma\notin \S^m_+
	\end{cases}.	
\end{equation}
and so $\dom(F_A)=\S^m_+$. In example (B):
\begin{equation}\label{eq:B}
\phi_B(x)\;=\;\begin{cases}
		\tfrac12 x^2 & x\geq 0\\
		+\infty & x< 0
	\end{cases},\qquad F_B(\Sigma)\;=\;\begin{cases}
		\tfrac12\tr({\Sigma^2}) & 
		\Sigma\in \overline\S^m_+\\
		+\infty & \Sigma\notin \overline\S^m_+
	\end{cases}.
\end{equation}
and so $\dom(F_B)=\overline\S^m_+$. Note that $\lim_{x\to 0^+}\phi_B(x)=0$ so $x=0$ lies in the domain of $\phi_B$. Here the extension was rather trivial because $\tfrac12 x^2$ is well defined for all $x\in \R$. Example (C) is slightly more subtle:
\begin{equation}\label{eq:C}
\phi_C(x)\;=\;\begin{cases}
		-x(1-\log(x)) & x> 0\\
		0 & x=0\\
		+\infty & x< 0
	\end{cases}, \quad F_C(\Sigma)\;=\;\begin{cases}
		-\tr(\Sigma-{\Sigma}\log(\Sigma)) & 
		\Sigma\in \S^m_+\\
		+\infty & \Sigma\notin \overline\S^m_+
	\end{cases}
\end{equation}
and so $\dom(F_C)=\overline\S^m_+$. Note that we did not write explicitly what is the value of the map $F_C$ on the boundary of $\overline \S^m_+$. This function, similarly to the pseudoinverse, takes the spectral decomposition $\Sigma=U\Lambda U^\top$ and applies the transformation $\phi_C$ only to the non-zero eigenvalues in $\Lambda$ leaving the zero eigenvalues unchanged. In particular it is well defined on $\overline \S^m_+$. Finally, in example (D):
\begin{equation}\label{eq:D}
\phi_D(x)\;=\;\begin{cases}
		\tfrac1x & x> 0\\
		+\infty & x\leq  0
	\end{cases},\qquad F_D(\Sigma)\;=\;\begin{cases}
		\tr({\Sigma^{-1}}) & 
		\Sigma\in \S^m_+\\
		+\infty & \Sigma\notin \S^m_+
	\end{cases}
	\end{equation}
and so $\dom(F_D)=\S^m_+$.

We now proceed to compute the convex conjugates of our running examples. In example (A), $\phi_A$ is given in \eqref{eq:A} and      
$$\phi^*_A(y)\;=\;\sup_{x\in \R}\{xy-\phi(x)\}\;=\;\sup_{x>0}\{xy+\log(x)\}\;=\;\begin{cases}
	-1-\log(-y) & \mbox{if } y<0,\\
	+\infty & \mbox{otherwise}.
\end{cases} $$ 
This calculation shows that 
$$
F^*_A(L)=\begin{cases}
	-m-\log\det(-L) &\mbox{if } L\in -\S^m_+,\\
	+\infty & \mbox{otherwise}
\end{cases},\qquad \dom(F_A^*)\;=\;-\S^m_+\;=\;\nabla F_A(\S^m_+). 
$$
Similarly, $\phi_B$ is given in \eqref{eq:B} and
$$\phi_B^*(y)\;=\;\sup_{x\geq 0}\{xy-\tfrac12 x^2\}\;=\;\begin{cases}
	\tfrac12 y^2 & \mbox{if } y\geq 0,\\
	0 & \mbox{otherwise}.
\end{cases}	
$$
In particular, 
\begin{equation}\label{eq:B2}
F^*_B(L)=\begin{cases}
	\tfrac12 \tr(L^2) &\mbox{if } L\in \overline\S^m_+,\\
	0 & \mbox{otherwise}. 
\end{cases},\qquad \dom(F_B^*)\;=\;\S^m\;\supset\;\nabla F_B(\S^m_+)=\S^m_+
\end{equation}
In example (C), $\phi_C$ is given by \eqref{eq:C} and 
$$\phi^*_C(y)\;=\;\max\{0,\sup_{x> 0}\{xy+x-x\log(x)\}\}\;=\;e^y.
$$ 
In particular, 
$F_C^*(L)=\exp(L)$ and $\dom(F_C^*)=\S^m=\nabla F_C(\S^m_+)$. Finally, $\phi_D$ is given in \eqref{eq:D} and its conjugate is
$$\phi^*_D(y)\;=\;\sup_{x> 0}\{xy-\tfrac{1}{x}\}\;=\;\begin{cases}
	-2\sqrt{-y} & \mbox{if }y\leq 0\\
	+\infty & \mbox{otherwise}
\end{cases}
$$ 
and we have
$$
F^*_D(L)\;=\;\begin{cases}
	-2\tr(\sqrt{-L}) & \mbox{if } L\in -\overline{\S}^m_+\\
	+\infty & \mbox{otherwise}.
\end{cases},\qquad \dom(F_D^*)\;=\;-\overline{\S}^m_+\;\supset\;\nabla F_D(\S^m_+).
$$

\section{Proofs from Section~\ref{sec:stat}}\label{app:proofsstats}

\begin{proof}[Proof of Theorem~\ref{th:clasas}]
By Theorem~4 in \cite{niemiro1992asymptotics}
$$
\sqrt{n}(\widehat\theta_n-\theta_0)\;=\;-\cI_0^{-1} \sqrt{n}\nabla M_n(\theta_0)+o_P(1).
$$
It is then enough to show that $\sqrt{n}\nabla M_n(\theta_0)\overset{d}{\rightarrow}N_m(0,\Omega)$. Using \eqref{eq:h}, we get $$\nabla M_n(\theta_0)\;=\;\tfrac1n \sum_i h(\theta_0,X_i)\;=\;[\<\Sigma_0-S_n,A_i\>]_{i=1}^d.$$ Moreover,$$
\E h(\theta_0,X)\;=\;[\<\Sigma_0-\E XX^\top,A_i\>]_{i=1}^d\;=\;0.
$$
Since the fourth moments exist, the central limit theorem gives that  $\sqrt{n}\nabla M_n(\theta_0)$ converges in distribution to $N(0,\Omega)$, where 
$$
\Omega_{ij}\;=\;n\E (\<\Sigma_0-S_n,A_i\>\cdot \<\Sigma_0-S_n,A_j\>)\;=\;n {\rm var}(S_n)[A_i,A_j]\;=\;\cS[A_i,A_j].
$$
 \end{proof}

\begin{proof}[Proof of Lemma~\ref{lem:finite}]
Recall the definition of $M_n(\theta)$ in \eqref{eq:Mn} as well as $\overline M$ and $\overline M_n$. We have
$$
M_n(\theta_0+\omega)-M_n(\theta_0)=\overline M(\omega)-\tr((S_n-\Sigma_0)L(\omega)).
$$
and $\tr((S_n-\Sigma_0)L(\omega))=\<\omega,\Delta^{(n)}\>$, which  gives
$$
\overline M_n(\omega)\;=\;\overline M(\omega)-\<\omega,\Delta^{(n)}\>.	
$$
The optimum of $\overline M_n(\omega)$ is then obtained for $\widehat\omega=\widehat\theta_n-\theta_0$ satisfying the first order condition $\nabla\overline M(\widehat\omega)=\Delta^{(n)}$. We get
$$\P(\|\widehat\theta_n-\theta_0\|\leq \epsilon)\;=\;\P(\Delta^{(n)}\in \nabla\overline M(\epsilon\B_2)),
$$
where $\epsilon \B_2=\{\omega: \|\omega\|\leq \epsilon\}$. By definition of $\mu$-strong convexity, for all $\|\omega\|\leq \epsilon$
 $$
\overline M(0)\;\geq\;\overline M(\omega)+\<\nabla\overline M(\omega),-\omega\>+\tfrac{\mu}{2}\|\omega\|^2.
$$
Using the fact that, $\overline M(0)=0$ and $\overline M(\omega)\geq 0$, we get that 
\begin{equation}\label{eq:kappaaux}
\<\nabla\overline M(\omega),\omega\>\;\geq\; \overline M(\omega) +\frac{\mu}{2}\|\omega\|^2\;\geq  \;\frac{\mu}{2}\|\omega\|^2.	
\end{equation}
By the H\"{o}lder's inequality $\<\nabla\overline M(\omega),\tfrac{\omega}{\|\omega\|_1}\>\leq \|\nabla\overline M(\omega)\|_\infty$. Thus, dividing \eqref{eq:kappaaux} by $\|\omega\|_1$ we get that
$$
\|\nabla\overline M(\omega)\|_\infty\;\geq \;\frac{\mu\|\omega\|^2}{2\|\omega\|_1}\;\geq\;\frac{\mu\|\omega\|}{2\sqrt{d}},
$$
where the last inequality follows from the fact that $\|\omega\|_1\leq \sqrt{d}\|\omega\|$. Note that, for every $\omega\in \R^d$, $\nabla \overline M(t \omega)$ is monotonely increasing in $t>0$. It now directly follows that 
\begin{equation}\label{eq:kappabound}
\P\big(\Delta^{(n)}\in \nabla\overline M(\epsilon \B_2)\big)\;\;\geq\; \;\P\big(\|\Delta^{(n)}\|_\infty\leq \inf_{\|\omega\|= \epsilon}\|\nabla\overline M(\omega)\|_\infty\big)\;\geq\;\P\big(\|\Delta^{(n)}\|_\infty\leq \tfrac{\mu\epsilon}{2\sqrt{d}}\big).	
\end{equation}
\end{proof}

\begin{proof}[Proof of Theorem~\ref{th:finite}]
By Lemma~\ref{lem:finite} and the union bound, we get
\begin{equation}\label{eq:auxfinite1}
\P(\|\widehat\theta_n-\theta_0\|> \epsilon)\;\leq\; \P(\|\Delta^{(n)}\|_\infty>\tfrac{\mu\epsilon}{2\sqrt{d}})\;\leq\;\sum_{k=1}^d \P(|\Delta^{(n)}_k|>\tfrac{\mu\epsilon}{2\sqrt{d}}).	
\end{equation}
We will bound each term on the right of \eqref{eq:auxfinite1}. Given $\Sigma_0$ and $A_1,\ldots,A_d\in \S^m$ denote the eigenvalues of $\sqrt{\Sigma_0}A_k\sqrt{\Sigma_0}$ by
$$
\lambda_{jk}\;=\;\lambda_j(\sqrt{\Sigma_0}A_k\sqrt{\Sigma_0})\;\in \;\R.
$$
	If $X\sim N(0,\Sigma_0)$ then  
	$X^\top A_k X\;=\;\sum_{j=1}^m \lambda_{jk} Z_j^2$, 
	where $Z\sim N_m(0,I_m)$. Thus, for every $k$ we have
	$$
	\Delta^{(n)}_k\;=\; \frac1n \sum_{i=1}^n (X_i^\top A_k X_i-\tr(\Sigma_0 A_k))\;=\;\frac1n\sum_{i=1}^n\sum_{j=1}^m \lambda_{jk}(Z_{ij}^2-1), 
	$$
	where $Z_{ij}\sim N(0,1)$ are all independent of each other. Recall from Definition~2.7 in \cite{wainwright2019high} that a random variable $Y$ with mean zero is sub-exponential with parameters $(\nu,\alpha)$, both positive, if $\E e^{\lambda Y}\leq e^{\nu^2\lambda^2/2}$ for all $|\lambda|<\tfrac{1}{\alpha}$. By Example~2.8 in \cite{wainwright2019high}, all $Z^2_{ij}-1$ are independent sub-exponential random variables with parameters $(\nu,\alpha)=(2,4)$. We also use another standard result: if $Y_1,\ldots,Y_N$ are independent sub-exponential with parameters $(\nu_i,\alpha_i)$, then for any fixed vector $u\in \R^N$, the linear combination $\sum_i u_i Y_i$ is sub-exponential with parameters $(\sqrt{\sum_i u_i^2\nu_i^2},\max_i \alpha_i)$; see p. 29 in \cite{wainwright2019high}. 	Since each  $\Delta^{(n)}_k$ is a linear combination of independent sub-exponential variables, we conclude that it is is sub-exponential with parameters $(\nu_k,\alpha_k)$ with
	$$
	\nu_k\;=\;\sqrt{\sum_{i=1}^n \sum_{j=1}^m  \tfrac{4}{n^2} \lambda_{jk}^2}\;=\;\frac{2}{\sqrt{n}}\sqrt{\sum_{j=1}^m \lambda^2_{jk}}\;=\;\frac{2}{\sqrt{n}}\|\Sigma_0 A_k\|_F .
	$$
	and
	$$
	\alpha_k\;=\;\frac4n \max_{j=1,\ldots,m} |\lambda_{jk}|\;=\;\frac4n \|\Sigma_0 A_k\| .
	$$
Directly from the definition, if $Y$ is sub-exponential with parameters $(\nu_k,\alpha_k)$, then it is sub-exponential with parameters $(\nu,\alpha)$ whenever $\nu_k\leq \nu$ and $\alpha_k\leq \alpha$. Using the fact that $A_1,\ldots,A_d$ for an orthonormal basis of $\cL$, we now show that, for each $k$,
	$$
	\nu_k\;\leq\;\frac{2}{\sqrt{n}}\|\Sigma_0\|,\qquad \alpha_k\;\leq\; \frac{4}{n}\|\Sigma_0\|.
	$$
	 Note that $\|\Sigma_0 A_k\|\leq \|\Sigma_0\| \|A_k\|\leq \|\Sigma_0\|$ simply because $\|A_k\|\leq \|A_k\|_F=1$. Moreover, $\|\Sigma_0 A_k\|_F\leq \|\Sigma_0\| \|A_k\|_F\leq \|\Sigma_0\|$. The fact that $\|\Sigma_0 A_k\|_F$ can be upper bounded by $\|\Sigma_0\| \|A_k\|_F$ follows easily from the first inequality in \cite{fang1994inequalities}, which states that if $A,B\in \overline\S^m_+$ then $\tr(AB)\leq \|A\|\tr(B)$. From this, for each $k$,
 $$
 \|\Sigma_0 A_k\|_F^2\;=\;\tr(\Sigma_0^2  A_k^2)\;\leq\; \|\Sigma_0^2\| \tr(A_k^2)\;=\;\|\Sigma_0\|^2 \|A_k\|_F^2\;\leq\; \|\Sigma_0\|^2.
 $$
We conclude that each $\Delta^{(n)}_{k}$ is a sub-exponential random variable with parameters $(\tfrac{2}{\sqrt{n}}\|\Sigma_0\|,\tfrac{4}{n}\|\Sigma_0\|)$. Standard sub-exponential tail bound (Proposition~2.9 in \cite{wainwright2019high}) imply then that
	$$
	\P(|\Delta^{(n)}_k|\geq t)\;\leq\; \begin{cases}
		2e^{-\frac{t^2n}{8\|\Sigma_0\|^2}} & \mbox{if }0\leq t\leq \|\Sigma_0\|,\\
		2e^{-\frac{tn}{8\|\Sigma_0\|}} & \mbox{otherwise }.
	\end{cases}
	$$
	The final bound follows by taking $t=\tfrac{\mu\epsilon}{2\sqrt{d}}$ and using \eqref{eq:auxfinite1}. For the last part of the theorem: 
\begin{equation}\label{aux:epsilon}	
	\delta=2d\exp\left\{-\frac{\mu^2\epsilon^2 n}{32d\|\Sigma_0\|^2}\right\}\quad\Longleftrightarrow\quad \epsilon=\frac{4\|\Sigma_0\|}{\mu}\sqrt{\frac{2d}{n}\log\left(\frac{2d}{\delta}\right)}.
\end{equation}
	If $n\geq 8\log(\tfrac{2d}{\delta})$ then the parameter $\epsilon$ in \eqref{aux:epsilon} satisfies $\epsilon\leq \tfrac{2\|\Sigma_0\|\sqrt{d}}{\mu}$, which, by the first part, implies that $\P(\|\widehat\theta_n-\theta_0\|>\epsilon)\leq \delta$ concluding the proof.
\end{proof}

\begin{proof}[Proof of Proposition~\ref{prop:kappabound}]
By Lemma~\ref{lem:gomega}, $\overline M(0)=0$ and $\nabla \overline M(0)=0$. Then by Taylor's theorem, for some $t\in [0,1]$,
$$
 \overline M(\omega)\;=\;\tfrac12 \omega^\top \nabla^2 \overline M(t\omega) \omega\;\geq\;\tfrac12 \lambda_{\min}(\nabla^2 \overline M(t\omega))\cdot \|\omega\|^2,
$$ 
where we will use the standard variational definition of the minimal eigenvalue of $\nabla^2\overline M(t\omega)$
$$
\lambda_{\min}(\nabla^2 \overline M(t\omega))\;=\;\min_{\|u\|=1}u^\top \nabla^2 \overline M(t\omega)u\;=\;\min_{\|u\|=1}\nabla^2 F^*(L(\theta_0+t\omega))[L(u),L(u)].
$$
Note that if $\|u\|= 1$ then $\|L(u)\|_F=1$ because
$$
\<L(u),L(u)\>\;=\;\sum_{i,j}u_i u_j\<A_i,A_j\>\;=\;\sum_{i=1}^m u_i^2\;=\;1.
$$ 
In other words, the minimal eigenvalue of $\nabla^2 \overline M(t\omega)$ corresponds to the restricted minimal eigenvalue of $\nabla^2 F^*(L(\theta_0+t\omega))$, which we bound here by the minimal unrestricted eigenvalue
$$
\lambda_{\min}(\nabla^2 \overline M(t\omega))\;=\;\min_{\|H\|_F=1, H\in \cL}\nabla^2 F^*(L(\theta_0+t\omega))[H,H]\;\geq \; \min_{\|H\|_F=1}\nabla^2 F^*(L(\theta_0+t\omega))[H,H].
$$
Since the gradient mappings $\nabla F$ and $\nabla F^*$ are inverses of each other on $\S^m_+$, by the inverse mapping theorem, at each point $L=\nabla F(\Sigma)\in \L^m_+$, the linear mapping $\nabla^2 F^*(L)$ is the inverse of the linear mapping $\nabla^2 F(\Sigma)$. Since these linear mappings are positive definite (by strict convexity of $F$ and $F^*$) it follows that 
$$
\min_{\|H\|_F=1}\nabla^2 F^*(L(\theta_0+t\omega))[H,H]\;=\;\frac{1}{\max_{\|H\|_F=1}\nabla^2 F(\Sigma(\theta_0+t\omega))[H,H]}.
$$ 
The proof of this result follows like in the standard matrix case by solving the corresponding Lagrange problems.  By Proposition~3.1 in \cite{juditsky2008large}, if $F(\Sigma)=\tr(\phi(\Sigma))$ and the condition \eqref{eq:techcond} holds, then , for every $\Sigma\in \S^m_+$ 
\begin{equation}\label{eq:boundnabla2}
\nabla^2 F(\Sigma)[H,H]\;\leq\;\ell\,\tr(H \phi''(\Sigma)H)\;=\;\ell \tr(\phi''(\Sigma)H^2).
\end{equation}
Since $\phi$ is strictly convex, $\phi''(x)>0$ for all $x>0$ and so, if $\Sigma\in \S^m_+$, then $\phi''(\Sigma)\in \S^m_+$. The first inequality in \cite{fang1994inequalities} states that if $A,B\in \overline\S^m_+$ then $\tr(AB)\leq \|A\|\tr(B)$. We use it with $A=\phi''(\Sigma)$ and $B=H^2$ to conclude
$$
\tr(\phi''(\Sigma)H^2)\;\leq \;\|\phi''(\Sigma)\|\tr(H^2)\;=\;\|\phi''(\Sigma)\|\|H\|^2_F\;=\;\|\phi''(\Sigma)\|
$$
whenever $\|H\|_F= 1$. Putting this all together gives
$$\lambda_{\min}(\nabla^2 \overline M(t\omega))\;\geq\; \frac{1}{\ell\, \|\phi''(\Sigma(\theta_0+t\omega))\|}\;\geq\;\frac{1}{\ell\, \alpha},$$
which shows that $\overline M(\omega)$ is $\tfrac{1}{\ell\alpha}$-strongly convex in the $\epsilon$-ball.
\end{proof}

%

\bibliographystyle{imsart-nameyear}
\bibliography{bib_cov}

\end{document}